\documentclass[10pt,twoside,a4paper]{amsart}

\usepackage[english]{babel}
\usepackage[utf8]{inputenc}

\usepackage{amsfonts,amsthm,amssymb,amsmath,amsthm,latexsym,wasysym,stmaryrd,mathrsfs,dsfont,txfonts,xcolor}
\usepackage{accents,multirow,rotating,subfigure,graphicx,bigstrut,lscape,multicol,enumerate,accents,mathtools,exscale}

\usepackage[normalem]{ulem}

\usepackage{pdftricks}
\usepackage{color}
\usepackage[pdftex,all]{xy}

\usepackage{hyperref}
\usepackage{appendix}
\usepackage[justification=centering]{caption}

\graphicspath{{Figures/}}

\numberwithin{equation}{section}

\theoremstyle{Theorem}

\newtheorem {theo}{Theorem}[section]
\newtheorem*{theo*}{Theorem}
\newtheorem {lemma}[theo]{Lemma}
\newtheorem*{lemma*}{Lemma}
\newtheorem {prop}[theo]{Proposition}
\newtheorem*{prop*}{Proposition}
\newtheorem {cor}[theo]{Corollary}
\newtheorem*{cor*}{Corollary}

\newtheorem*{cor_proof*}{Corollary (of the proof)}

\newtheorem*{conjecture*}{Conjecture}
\theoremstyle{definition}
\newtheorem {defi}[theo]{Definition}
\newtheorem*{defi*}{Definition}
\newtheorem {nota}[theo]{Notation}
\newtheorem*{nota*}{Notation}
\theoremstyle{remark}
\newtheorem {remark}[theo]{Remark}
\newtheorem*{remark*}{Remark}

\newtheorem*{warning*}{Warning}

\newtheorem*{remarques*}{Remarks}

\newtheorem*{warnings*}{Warnings}

\newtheorem*{convention*}{Convention}

\newtheorem*{exemple*}{Example}

\newtheorem*{exemples*}{Examples}

\newtheorem*{question*}{Question}

\newtheorem*{questions*}{Questions}

\newtheorem*{fact*}{Fact}
\newtheorem{claim}[theo]{Claim}
\newtheorem*{acknowledgments}{Acknowledgments}

\def\C{{\mathcal C}}
\def\CC{{\mathcal C}}

\def\R{{\mathds R}}

\def\Z{{\mathds Z}}

\def\2Z{{\fract{\Z}{2\Z}}}

\newcommand{\fract}[2]{\hbox{\leavevmode 
\kern.1em \raise .25ex \hbox{\the\scriptfont0 $#1$}\kern-.1em }\big/
  {\hbox{\kern-.15em \lower .5ex \hbox{\the\scriptfont0 $#2$}} }}
\newcommand{\subfract}[2]{\hbox{\leavevmode
  \kern0em \raise .25ex \hbox{\the\scriptfont0 \tiny $#1$}\kern-.1em }/
  {\hbox{\kern-.15em \lower .5ex \hbox{\the\scriptfont0 \tiny $#2$}} }}

\newcommand{\dessin}[2]{
  \vcenter{\hbox{\includegraphics[height=#1]{#2.pdf}}}}

\renewcommand{\quote}[1]{`#1'}

\def\nR{\textnormal R}

\newcommand{\ve}{\varepsilon}
\newcommand{\si}{\sigma}

\definecolor{pink}{rgb}{0.858, 0.188, 0.478}
\definecolor{orange}{rgb}{1, 0.647, 0}

\begin{document} 

\title{Higher order Kirk invariants of link maps} 
\author[B. Audoux]{Benjamin Audoux}
         \address{Aix Marseille Univ, CNRS, Centrale Marseille, I2M, Marseille, France}
         \email{benjamin.audoux@univ-amu.fr}
\author[J.B. Meilhan]{Jean-Baptiste Meilhan} 
\address{Univ. Grenoble Alpes, CNRS, Institut Fourier, F-38000 Grenoble, France}
	 \email{jean-baptiste.meilhan@univ-grenoble-alpes.fr}
\author[A. Yasuhara]{Akira Yasuhara} 
\address{Faculty of Commerce, Waseda University, 1-6-1 Nishi-Waseda,
  Shinjuku-ku, Tokyo 169-8050, Japan}
	 \email{yasuhara@waseda.jp}
%
%
\begin{abstract} 
We define numerical link-homotopy invariants of link maps of any number of components, which naturally generalize the Kirk invariant.
The Kirk invariant is a link-homotopy invariant of $2$-component link maps given by linking numbers of loops based at self-singularities of each component with the other spherical component; 
our invariants use instead ingredients from Milnor's higher order link invariants, and are extracted from the reduced fundamental groups of the exterior. 
We provide practical algorithms to compute these invariants from an appropriate cross-section, as well as families of examples that are therewith detected.
The main proofs use the combinatorial theory of cut-diagrams previously developed by the authors.
\end{abstract} 

\maketitle

\section{Introduction} 
A \emph{link map} is a continuous map from a disjoint union of spheres (possibly of various dimensions) to the $n$-dimensional sphere ($n\ge 0$), with pairwise disjoint images. 
The natural equivalence relation on link maps is \emph{link-homotopy}, that is homotopies through link maps --- a notion that was first introduced by Milnor in \cite{Milnor2} in the study of links in $3$-space.
The  study of link maps was initiated by Scott \cite{Scott} and Massey-Rolfsen \cite{MR}, in codimension  larger than $2$. 
Koschorke showed that, in a large metastable range, link-homotopy of link maps essentially reduces to problems on higher homotopy groups of spheres \cite{Koko}; see also e.g.  \cite{Koko2,Koko3}.

This paper is concerned with the study of link maps of $2$-dimensional spheres in $S^4$, which turns out to show rather different behaviors. 
From now on, the term \lq link map\rq\, will always implicitly refer to this codimension $2$ situation. 
The first step in this study was the work of Fenn and Rolfsen, who constructed a $2$-component link map which is not link-homotopically trivial \cite{FR}.  
Kirk defined in the late eighties a link-homotopy invariant of $2$-component link maps, as follows.
Let $f : S^2_1 \cup S^2_2\rightarrow S^4$ be a link map, which we can freely assume to have finitely many self-transverse singular points.\footnote{In what follows, we will always implicitly assume that all link maps are in general position.} 
For each double point $p$ in the double point set $P_1$ of $f(S^2_1)$, pick a simple loop $\alpha_p$ on $f(S^2_1)$ based at $p$, such that $f^{-1}(\alpha_p)$ is connected, and denote by $n_p=\left\vert \textrm{lk}\left(\alpha_p,f(S^2_2)\right)\right\vert$ the absolute value of the linking number of $\alpha_p$ with the second component. Then Kirk defines 
\begin{equation}\label{eq:kirk1}
 \sigma_1(f) := \sum_{p\in P_1} \ve(p)(t^{n_p} -1)\in \Z[t], \\[-0.15cm]
 \end{equation}
\noindent where $\ve(p)$ denotes the sign at the intersection at $p$. 
Reversing the roles of the components, we similarly define $\sigma_2(f)\in \Z[t]$. 
The \emph{Kirk invariant of $f$} is the pair $\si(f):=\left(\si_1(f),\si_2(f)\right)\in \Z[t]\oplus \Z[t]$, see \cite{Kirk}. 
Kirk showed that $\si$ is a link-homotopy invariant.  
Strikingly, this invariant is in fact a complete invariant of $2$-component link maps up to link-homotopy, as showed thirty years later by Schneiderman and Teichner  \cite{ST}. 
For link maps of any number of components, a similar invariant can  easily be defined, detecting similar pairwise linking phenomena among $2$-component sub-link maps; see e.g. \cite{Li}. 
\medskip 

The purpose of the present paper is to define numerical invariants of
link maps which can detect triple and higher linking
  phenomena. They can be seen as higher order Kirk invariants. 

This builds on a very elementary observation about Kirk's definition : taking the absolute value $n_p$ of the linking number, 
amounts to specifying a \emph{preferred orientation} on the loop $\alpha_p$, which has \emph{positive} linking number with the other link map component; 
in other words, this defines a \lq positive\rq\, element of $H_1(S^4\setminus f(S^2_2))$ representing the loop $\alpha_p$ based at $p$. 

The idea for our higher order invariants can then be roughly
summarized as follows --- complete definitions will be given in Section \ref{sec:def}. 
First, since we seek for higher order link-homotopy invariants, we will consider elements in the \emph{reduced} fundamental group of the link map exterior, rather than just its first homology group; this notion was introduced by Milnor in his seminal work on link-homotopy for links in the $3$-sphere \cite{Milnor1,Milnor2}. Next, we introduce a notion of \emph{positive} element in this reduced group, which allows us to specify a preferred orientation on any path in a link map complement. 
Finally, in order to extract numerical invariants, we consider the \emph{reduced Magnus expansion} of these positive elements; this yields polynomials in non-commuting variables, with integral coefficients, and taking these coefficients modulo a suitable indeterminacy, as a gcd of certain lower order coefficients, provides us with the desired invariants. See Theorems \ref{th:main} and \ref{th:main2} for precise statements.  
We stress that a \emph{basing} for the link map, which is
a choice of a meridian for each component, is chosen and fixed  in
this construction, but that our numerical invariants are independent
of this choice. 
 
As discussed in Section \ref{sec:covering}, the Kirk invariant has a
natural reformulation in terms of covering spaces, which extends
to higher order link map invariants of any number of components. 
This was already observed by Stirling in his recent paper \cite{stirling23}, which summarizes the work of his 2022 PhD thesis \cite{Stirling}.
This construction naturally yields invariants of \emph{based} link maps and Stirling made, in the $3$-component case, a remarkable work to determine the maximal quotient that is basing independent. As a matter of fact, his invariant stands as the best known candidate for a complete invariant of $3$-component link maps; but the techniques of \cite{stirling23} turn out to be rather intricate for higher numbers of components, see Remark \ref{rem:stirling}.
In this paper, we use instead a combinatorial approach, based on the theory of \emph{cut-diagrams} developed by the authors in \cite{AMY}. 
This will not only allow for concrete computational examples at any order, but also provide a self-contained framework for all proofs. 
Our approach in terms of cut-diagrams seems to have certain advantages over topological arguments. 
First, the construction is purely combinatorial, making the proof of the theorem rather elementary. 
Moreover,  this makes explicit computations accessible at any order: we provide concrete examples of link maps of any number of components realizing our invariant in Section \ref{sec:real}. 
We also show in Section \ref{sec:compute} that there is a surjective
map from a certain set of singular links in $S^3$ to the set of link
maps, and that higher order link-homotopy invariants of a link map can be calculated by a practical algorithm from 
a singular link which is a preimage under this map. 
 Another advantage of our approach is that it 
generalizes to \emph{surface-link maps} of any number of components, which are continuous maps from a disjoint union of surfaces to the $4$-dimensional sphere, with pairwise disjoint images; see Section \ref{sec:surfaces}.

\begin{acknowledgments}
The second author would like to thank Paul Kirk for stimulating discussions regarding \cite{Kirk} during his stay at the Institut Fourier in the summer 2022. The authors are
also grateful to Mark Powell for bringing the PhD thesis \cite{Stirling} to their knowledge during the
process of this work, and for useful discussions.  
The first, resp. second, author is partially supported by the project SyTriQ (ANR-20-CE40-0004), resp. the project AlMaRe (ANR-19-CE40-0001-01), of the ANR.
The third author is supported by the JSPS KAKENHI grant 21K03237.
\end{acknowledgments}

\section{Definition of the higher order Kirk invariants} \label{sec:def}
Let $f:S_1^2\cup\cdots \cup S_n^2\longrightarrow S^4$ be a link map, that is, a continuous map with pairwise disjoint images.
Denote by $L=K_1\cup\cdots\cup K_n=f(S_1^2)\cup\cdots\cup f(S_n^2)$ the image of $f$, and for each $i\in\{1,\cdots,n\}$ set 
\[L_i:=L\setminus K_i,\quad M_i=S^4\setminus L_i\quad \textrm{ and}\quad G(L_i)=\pi_1(M_i,x_0), \]
for some basepoint $x_0$ in the exterior of $L$.
We note that, since $f$ can freely be assumed to be an immersion with
transverse double points, the oriented image $L$ of $f$, with ordered components, may safely be identified with the link map $f$ itself; in the rest of the paper, we shall freely call $L$ a link map as well. 

\subsection{Reduced groups and positive elements}\label{sec:pos}

Here, and throughout the rest of this paper, given two group elements $x$ and $y$, we use the convention $[x,y]=x^{-1}y^{-1}xy$ and $x^y=y^{-1}xy$. 

\begin{defi} 
A  \emph{meridian} for the $i$th component of $L$, or simply \emph{$i$th meridian}, is a loop in $S^4\setminus L$, of the following form. 
Pick a point $x_i$ on the $i$th component $K_i$ of $L$, and a small disk $D_i$ intersecting $L$ transversely at $x_i$; pick also a path $\gamma_i$  running from $x_0$ to the boundary of $D_i$. Then an $i$th meridian is given by the loop $\gamma_i (\partial D_i) \gamma_i^{-1}$, oriented in such a way that it has linking number one with $K_i$.  
A \emph{basing} for $L$ is a choice of meridian for each component, which we may assume to be mutually disjoint except at $x_0$. 
\end{defi}

In what follows, we assume that a choice of basing has been made for our link map $L$. 

\begin{remark}\label{def:basing}
A basing thus specifies, for each component $K_i$ of $L$,  a point $x_i$ and a path running from $x_0$ to $x_i$, which we abusively also denote by $\gamma_i$. 
\end{remark}

The following notion was first introduced by Milnor in \cite{Milnor1}, and is thus sometimes called \emph{Milnor group}.
\begin{defi}\label{def:red} 
Given a group $G$ normally generated by $a_1,\cdots,a_m$, the \emph{reduced group} $\nR G$ is defined as the quotient of $G$ by the normal subgroup generated by commutators $[a_j,a_j^g]$ for all $j$ and all $g\in G$. 
\end{defi}

Since, for each $i$, the group $G(L_i)$ is normally generated by a choice of meridian for each component of $L_i$, we can consider the reduced group $\nR G(L_i)$. 
In fact, the following is known.
\begin{lemma}\label{lem:isor}
$\nR G(L_i)$ is isomorphic to $\nR F^i_{n-1}$, the reduced group of the free group $F^i_{n-1}$ on $n-1$ generators $\{x_1,\ldots,x_n\}\setminus \{x_i\}$.  
\end{lemma}
This isomorphism is specified by the basing for $L$, and maps the $j$th meridian of $L$ to $x_j$ for each $j$. 
Lemma \ref{lem:isor} seems to have been first observed by Krushkal, see Section 3.8 in \cite{Krushkal}.
A proof was also recently given in \cite{AMY} using cut-diagrams (resulting in a more general statement), and independently in \cite[Prop.~7.1]{stirling23} (see also \cite[Prop.~8.0.1]{Stirling}) in a purely algebraic way. 

Denote by $\Z\langle\langle X_1,\cdots,X_n \rangle\rangle$ the ring of formal power series in non-commuting variables $X_1,\cdots, X_n$. 
For each $i$, denote by $\Lambda_i$ the quotient ring of $\Z\langle\langle X_1,\cdots,X_n \rangle\rangle$ by the ideal generated by monomials containing $X_i$ and monomials containing twice a same variable. 
We may identify $\Lambda_i$ with the set of polynomials whose terms contains at most once each variable, and does not contain the variable $X_i$. 
The \emph{reduced Magnus expansion $E_i$} is an injective group homomorphism (see e.g. \cite[Prop.~7.10]{ipipipyura}) 
$$ E_i:\, \nR F^i_{n-1} \longrightarrow \Lambda_i $$
\noindent defined by $E_i(x_j^{\pm 1})=1\pm X_j$.
Since indices correspond to components, and components are ordered, monomials in $\Lambda_i $ are endowed with a total order, inherited from the lexicographic order on the indices of the variables. In particular, we can make sense of the \lq first non vanishing\rq\, term in some reduced Magnus expansion.

The following is a key ingredient in our construction (see Remark \ref{rem:eco}). 
\begin{defi}\label{def:positive}
An element of $\nR G(L_i)$, for some $i$, is \emph{positive}, 
if it is trivial or if the first non vanishing term of its reduced Magnus expansion $E_i$ has positive coefficient. 
\end{defi}

\begin{remark}\label{rem:positive}
The notion of positivity makes implicitly use of the the isomorphism from $\nR G(L_i)$  to $\nR F^i_{n-1}$, which is specified by the chosen basing for $L$. 
We stress, however, that this notion is in fact independent of the basing, that is, a positive element of $\nR G(L_i)$ for a given basing, is positive for any other basing choice. 
This follows from basic properties of the Magnus expansion as follows. 
A basing change on the $j$th component of $L$ amounts to substituting, in the reduced free group, the $j$th generator $x_j$ by a conjugate 
$g^{-1}x_jg$ from some $g\in \nR F^i_{n-1}$. Setting $E_i(g)=1+U \in \Lambda_i $, this substitution affects
the reduced Magnus expansion $E_i(x)$, for any $x\in \nR G(L_i)$, by replacing each occurrence of $X_j$ by $X_j+X_jU-U X_j$. 
This implies in particular that the first non vanishing term of $E_i(x)$ remains unchanged, which in turn implies our claim for positive elements. Note that for any nontrivial element $h\in \nR G(L_i)$, we have that $h$ is positive if and only if $h^{-1}$ isn't. 

Positivity however depends on the ordering of the link map components. 
\end{remark}

\begin{remark}\label{rem:orient}
An unoriented nontrivial loop in the exterior of $L_i$ represents a unique positive element in $\nR G(L_i)$. 
Hence the notion of positivity induces a preferred orientation of such a loop. 
This is simply because reversing the orientation, reverses the sign of all lowest degree terms in the corresponding Magnus expansion. 
\end{remark}

We conclude this section by setting some notation that will be used throughout the rest of this paper.
\begin{nota}\label{nota:coeff}
Let $W$ be an element of  $\Lambda_i $ for some index $i$ in $\{1,\cdots,n\}$. 
For any non-repeating sequence $I=i_1\cdots i_k$ of elements of $\{1,\cdots,n\}\setminus \{i\}$,  denote by $X_I$ the monomial 
 $$ X_I := X_{i_1}\cdots X_{i_k}. $$
We denote by $\kappa_W(I)$ the coefficient of $X_I$ in $W$. 
In other words, we have 
$$ W = \sum_{I} \kappa_W(I) X_{I}\in \Lambda_i, $$
\noindent where the sum runs over all non-repeating sequences $I$ of indices in $\{1,\cdots,n\}\setminus \{i\}$. 
Furthermore, we set
$$ D_W(I):=\textrm{gcd}\left\{ \kappa_W(J)~|~\mbox{$J$ is a subsequence of $I$}~;~J\neq I \right\}. $$
\end{nota}

\subsection{Higher order link-homotopy invariants of link-maps}\label{sec:define}

For a point $p$ in the singular set $P_i$ of $L_i$, 
choose an unoriented loop $\gamma_p$ based at $x_0$ as follows. 
First, consider the path $\gamma_i$ running from $x_0$ to the point $x_i$ on $K_i$ given by the chosen basing of $L$ (here we use the notation of Remark \ref{def:basing}). Next, pick a pair of paths $\alpha_1,\alpha_2$ running from $x_i$ to $p$, forming a loop $\alpha_1 \alpha_2^{-1}$ that changes branches at $p$ and avoids all other points in $P_i$. Set $\gamma_p:=\gamma_i\cdot \alpha_1\cdot \alpha_2^{-1}\cdot \gamma_i^{-1}$. 
\begin{nota}\label{nota:gp}
We denote by $g_p$ the unique positive element of $\nR G(L_i)$ given  by the loop $\gamma_p$. 
\end{nota}

Denote by $\phi$ the isomorphism from $\nR G(L_i)$ to the reduced free group $\nR F^i_{n-1}$, 
induced by the chosen basing of $L$ (Lemma \ref{lem:isor}).

For each component $i$, we set 
 \begin{equation}\label{eq:kirky}
  S_i(L) := \sum_{p\in P_i} \ve(p)(\phi(g_p) -1)\in \mathbb{Z}\nR F^i_{n-1}, 
  \\[-0.1cm]
 \end{equation}
\noindent where $\ve(p)$ denotes the sign of $p$, and where $1$ is the trivial element in $\nR F^i_{n-1}$. 

The following will be proved in Section \ref{sec:proof_linkmaps}.
\begin{theo}\label{th:main0}
For each $i$, $S_i(L)$ is a link-homotopy invariant of the \emph{based} link map $L$. 
Moreover, if $L$ is a link-homotopically trivial (unbased) link map, then $S_i(L)=0$ for all $i$. 
\end{theo}

These invariants naturally generalize the Kirk invariant. 
Indeed, when $n=2$, the reduced group $\nR G(L_i)=\nR G(L\setminus K_i)$ is isomorphic to $H_1(M_i)\cong \mathbb{Z}=\langle t\rangle$. Hence for $\{i,j\}=\{1,2\}$, 
an element $g_p\in \nR G(L_i)$, for some singular point $p$ of $K_i$, is a positive element if and only if it represents a loop based at $p$ having positive linking number with $L_j$.  
Note moreover that in the case $n=2$, the sum $S_i(L)$ is independent of the choice of basing, since $\nR F^1_1$ is abelian, a fact that no longer holds for $n\ge 3$. 
\medskip

In the next two subsections, we extract from (\ref{eq:kirky}) link-homotopy invariants of link maps, that do not depend on the basing.

\subsubsection{The numerical link-homotopy invariants $\tilde{\kappa}(I)$}\label{sec:kappa}

As hinted in Remark \ref{rem:positive}, changing the basing essentially amounts, in our construction, to replacing some generator of the reduced free group with a conjugate of it. This is reminiscent of the $3$-dimensional situation, where a similar phenomenon is handled in the definition of Milnor link invariants \cite{Milnor2}. We thus propose in this subsection an adaptation of Milnor's work to the present setting, that yields numerical invariants which are only well-defined modulo an indeterminacy arising from lower order invariants.

In order to extract numerical invariants from (\ref{eq:kirky}), that are independent of the chosen basing, we consider the reduced Magnus expansion $E_i(L) := E_i\left( S_i(L) \right)$ of this equation. (Note that $E_i$ extends naturally to the group ring $\mathbb{Z} \nR F^i_{n-1}$.)
This yields a finite sum (using Notation \ref{nota:coeff})
\begin{equation}\label{eq:Ekirky}
  E_i(L) := E_i\left( S_i(L) \right) = \sum_I \kappa_{E_i(L)}(I) X_I\in \Lambda_i, 
 \end{equation} 
\noindent where the sum ranges over all non-repeated sequences $I$ in $\{1,\ldots,n\}\setminus\{i\}$, for some integer coefficients $\kappa_{E_i(L)}(I)$. 
For any such sequence $I$, recall from Notation \ref{nota:coeff} that we set 
 \[ D_{E_i(L)}(I):=\gcd\left\{\kappa_{E_i(L)}(I)~|~J\mbox{ is a subsequence of }I~;~J\neq I\right\}.\]

Then we have the following theorem. 
\begin{theo}\label{th:main}
For each $i$, and any non-repeated sequence $I$ in $\{1,\ldots,n\}\setminus\{i\}$, the residue class $\widetilde{\kappa}_L(I;i)$ of $\kappa_{E_i(L)}(I)$ modulo $D_{E_i(L)}(I)$, 
is a link-homotopy invariant for $L$. 
In particular, $\widetilde{\kappa}_L(I;i)$ is independent of the basing.
\end{theo}

The proof of Theorem \ref{th:main}  is postponed to Section \ref{sec:proof_linkmaps}. 
Examples of computations will be given in Section \ref{sec:ex}. 

In practice, Theorem \ref{th:main} yields a family of numerical invariants, that are globally extracted from each component by considering the sum (\ref{eq:kirky}). 
It is natural to expect that link-homotopy invariants could be extracted from each positive element of the reduced group contributing to this sum.
This is the idea of the construction given in the next subsection.

\subsubsection{The link-homotopy invariants $\mathcal{K}_{i}$ and $\mathcal{K}(I;i)$}\label{sec:refine}

By simply grouping terms of same nature in Equation (\ref{eq:kirky}), it formally rewrites as 
\begin{equation}\label{eq:kirky2}
 S_i(L) = \sum_{g\in \nR G(L_i)^+} \rho_L(g)(\phi(g) -1)\in \mathbb{Z}\nR F^i_{n-1},
\end{equation}
for some integer coefficients $\rho_L(g)$, 
where  $\nR G(L_i)^+$ is the set  of positive elements in $\nR G(L_i)$. 
In other words, for any $g\in \nR G(L_i)^+$, the integer $\rho_L(g)$ is the sum of signs of all singular points $p\in P_i$ such that the associated loop $\gamma_p$ represents $g$. 
We note that these coefficients $\rho_L(g)$ have a very natural topological interpretation in terms of covering spaces, see Remark \ref{rem:rototo} below.

Since the reduced Magnus expansion $E_i$ is injective, 
$S_i(L)$ naturally corresponds to the multiset 
$ \left\{ \left(\rho_L(g),E_i\left(\phi(g)\right)\right)\, \vert \, \textrm{$g\in \nR G(L_i)^+$ positive element; $\rho_L(g)\neq 0$} \right\}$. 

One can use a similar method as in Section \ref{sec:kappa} to derive invariants that do not depend on the choice of basing, as follows.
Pick any  positive element $g$ in $\nR G(L_i)^+$ such that $\rho_L(g)\neq 0$. 
Taking the reduced Magnus expansion of $\phi(g)$, produces a sum 
\begin{equation}\label{eq:Ekirky_g}
  E_i(\phi(g)) = 1 + \sum_{I}   \kappa_{E_i(\phi(g))(I)} X_I\in \Lambda_i, 
\end{equation} 
\noindent where the sum ranges again over all non-repeated sequences $I$ in $\{1,\ldots,n\}\setminus\{i\}$. 
For any such sequence $I$, we denote by $\widetilde{\kappa}_L(I;g)$  the residue class of $\kappa_{E_i(\phi(g))}(I)$ modulo $D_{E_i(\phi(g))}(I)$  
(see Notation \ref{nota:coeff}); we stress that these residue classes are indexed by a sequence and the given positive element $g$. 
Consider the multiset 
 \[ \mathcal{K}_{i}(L) = \left\{ \left(\rho_L(g), \small{\sum}_{\large{I}} \widetilde{\kappa}_L(I;g)X_I \right)\, \vert \, \textrm{$g\in \nR G(L_i)^+$ ; $\rho_L(g)\neq 0$} \right\}, \]
where, as above, the sums are over all non-repeating sequences $I$ of
indices in $ \{1,\cdots,n\}\setminus \{i\}$. We stress that $\mathcal{K}_{i}(L)$ is not a set in general since, 
if two different positive elements $g$ and $g'$ give the same pair
$\left(\rho_L(g), \Sigma_I \widetilde{\kappa}_L(I;g) X_I\right)$, then each contributes to an element in 
$\mathcal{K}_{i}(L)$.  
We have the following.
\begin{theo}\label{th:main2}
  For each $i$, 
  $ \mathcal{K}_{i}(L)$ is a link-homotopy invariant for $L$. In particular, it is independent of the basing.
\end{theo}
The proof of Theorem \ref{th:main2} will be given in Section \ref{sec:proof_linkmaps}. 
As explained in Remark \ref{rem:cor}, this proof actually provides numerical link-homotopy invariants of link maps, as follows.
\begin{cor}\label{th:main3}
 For each $i$, and for any non-repeated sequence $I$ in $\{1,\ldots,n\}\setminus\{i\}$,  the multiset 
 $$ \mathcal{K}_L(I;i) = \left\{ \left(\rho_L(g), \widetilde{\kappa}_L(I;g) \right)\, \vert \, \textrm{$g\in \nR G(L_i)^+$ ; $\rho_L(g)\neq 0$} \right\} $$
 is a link-homotopy invariant for $L$.
 \end{cor}

Examples of computations for these invariants $\mathcal{K}_i$  and $\mathcal{K}(I;i)$ will be given in Section \ref{sec:ex}. 

\begin{remark}
For each component $K_i$ of the link map $L$, Theorem \ref{th:main2} provides a set of invariants associated with each positive element in $\nR G(L_i)$, while Theorem \ref{th:main} yields a unique family of numerical invariants for each component. 
We give in Lemma \ref{lem:Y1} an example of two link maps that cannot be distinguished by Theorem \ref{th:main}, but are shown to not be link-homotopic by Theorem \ref{th:main2}. 
However, there are also cases of link maps where the opposite
phenomenon occurs, that is, pairs of links maps that cannot be
distinguished by the invariants of Theorem \ref{th:main2}, but that
are detected by the numerical invariants of Theorem \ref{th:main}; a concrete example is given in Lemma \ref{lem:compare}.
\end{remark}

\subsection{Higher order Kirk invariants via covering spaces}\label{sec:covering}

Levine had observed \cite[Note 2]{Kirk} that the Kirk invariant of a link map $f : S^2_1 \cup S^2_2\rightarrow S^4$ admits the following natural definition. 
Denoting by $M_2$ the exterior $S^4\setminus f(S^2_2)$, one can consider the restriction map $f: S^2_1\rightarrow M_2$, which lifts to the infinite cyclic covering $\widetilde{M_2}$.   
Then 
\begin{equation}\label{eq:kirk}
 \sigma_1(L) = \sum_{k\ge 1} \left(f_0(S^2_1)\cdot f_k(S^2_1) \right)(t^k-1)\in \Z[t], 
\end{equation}
\noindent where $f_0(S^2_1)$ is a fixed lift of $f(S^2_1)$, with $k$th translate denoted by $f_k(S^2_1)$, and  where $\cdot$ is the geometric intersection. 

This alternative definition  extends very naturally to link maps with arbitrarily many components. 	
Let $f:S_1^2\cup\cdots \cup S_n^2\longrightarrow S^4$ be an $n$-component link map with image $L=K_1\cup\cdots\cup K_n=f(S_1^2)\cup\cdots\cup f(S_n^2)$. 
As above, denote by $L_i:=L\setminus K_i$, $M_i=S^4\setminus L_i$,  and $G(L_i)=\pi_1(M_i,x_0)$, for each $i\in\{1,\cdots,n\}$.
As a substitute for the infinite cyclic covering in the $2$-component case, we consider the covering space $\widetilde{M}_i$ of $M_i$ associated to 
the kernel of the natural projection $G(L_i)\longrightarrow \nR G(L_i)$. 
Note that the covering transformation group of $\widetilde{M}_i$ is isomorphic to $\nR G(L_i)$. 
Let $f_0:S^2_i\longrightarrow \widetilde{M}_i$ be a lift of 
the restriction map $f:S^2_i\longrightarrow M_i$. Then we have an equivariant 
intersection
\[\sigma_i(L):= \sum_{g\in \nR G(L_i)}(f_0(K_i)\cdot g(f_0(K_i))
(g-1)\in \mathbb{Z}\nR G(L_i),  \]
which essentially recovers (\ref{eq:kirk}) for $n=2$. More precisely, for $n=2$ and $i=1$, (\ref{eq:kirk}) represents \lq half\rq\, of the above sum, given by summing only over  \emph{positive} elements of $\nR G(L_i)$, see Remark \ref{rem:eco} below.

\begin{remark}\label{rem:rototo}
The geometric intersection numbers $f_0(S^2_i)\cdot g(f_0(S^2_i)$ in the defining equation for $\sigma_i(L)$, coincide for positive elements $g$ with the coefficients $\rho_L(g)$ introduced in Section \ref{sec:refine}.
\end{remark}

\begin{remark}\label{rem:eco}
It is natural to consider defining numerical link-homotopy invariants  of $L$ similar to those of Theorem \ref{th:main}, 
by considering the coefficients of the reduced Magnus expansion $E_i(\sigma_i(L))$.
Note however that the sum $\sigma_i(L)$ formally rewrites as  
\[ \sigma_i(L) = \sum_{g\in \nR G(L_i)^+}\rho_L(g) \left(\phi(g) + \phi(g^{-1}) - 2\right)
                     = - \sum_{g\in \nR G(L_i)^+}\rho_L(g) \left(\phi(g) -1\right)\left(\phi(g^{-1}) - 1\right).  \]
This formula suggests that numerical invariants extracted from the coefficients of $E_i(\sigma_i(L))$ 
might vanish while those constructed in the previous subsections do
not. We provide a concrete example in Remark \ref{rem:eko}, which illustrates the advantage of considering only \emph{positive} elements in our definition.
\end{remark}

\begin{remark}\label{rem:stirling}
Although quite natural, the above generalization $\sigma_i$ of the Kirk invariant seems to have only appeared in the recent preprint of Stirling \cite{stirling23},  which focusses on the $3$-compo\-nent case. 
It is not so hard to see that $\sigma_i(L)$ is a link-homotopy invariant for \lq based' link maps, see \cite[Prop.~5.3]{stirling23} and also \cite[Prop.~6.1.3]{Stirling}.  
It is however much more difficult in practice to derive from this construction an explicit and computable invariant of (non based) link maps. 
This is addressed by Stirling, who defined an
equivalence relation on $(\nR F_{n-1})^n$ that encompasses algebraically all possible basing changes; 
Stirling managed to make this equivalence relation explicit in the
$3$-component case. As a matter of fact, Stirling's invariant  detects the $3$-component link map from Remark \ref{rem:eko}.
\end{remark}

\section{Cut-diagrams}\label{sec:cutd}

Let us begin with a review of the theory of cut-diagrams, in a general setting.
Note that the definition of cut-diagrams given below, corresponds to the notion of \quote{self-singular cut-diagram} given in \cite[\S 7.3]{AMY}. 
Let $\Sigma$ be an oriented surface with $n$ connected components, possibly with boundary.

\subsection{Cut-diagrams and surface-link maps}\label{sec:cut_def}

Consider a compact oriented (generically, but not necessarily properly) immersed $1$-manifold $P$ in $\Sigma$. This splits $\Sigma$ into connected components called \emph{regions}. 
 Endow each transverse double point of $P$ with an over/under decoration, as in usual tangle diagrams, splitting $P$ into \emph{cut arcs}. 
 The resulting diagram on $\Sigma$ then contains crossings and univalent vertices. Univalent vertices in the interior of $\Sigma$ shall be of two types, being either black $\bullet$ or white $\circ$ dots. 

\begin{defi}\label{def:cut}
  A \emph{cut-diagram} over $\Sigma$ is obtained by labeling all cut arcs by regions, according to the following \emph{labeling rules}: 
 \begin{itemize}
  \item[1)] for each crossing, involving labels $A,B,C$ as shown on the left-hand side of Figure \ref{fig:cond}, the regions $A$ and $B$ are adjacent along a $C$-labeled cut arc as illustrated in the figure; 
  \item[2)] a cut arc containing a $\bullet$ dot in region $A$, is labeled by $A$ (see the center of Figure \ref{fig:cond});
  \item[3)] for each connected component of $\Sigma$, the set of
    $\circ$ dots is endowed with a partition into $2$-element subsets. Each subset $\{p_1,p_2\}$ is such that, if $p_1$ lies in some region $X$ and has incident cut arc labeled by $Y$,
    then $p_2$ lies in $Y$ with incident cut
    arc labeled by $X$ 
    (see the right-hand side of Figure \ref{fig:cond}), and the local orientations at $p_1$ and $p_2$ are the same; we call the subset $\{p_1,p_2\}$, an \emph{$(X,Y)$-pair}. 
 \end{itemize}
 \end{defi}
 
Note that by definition, for each $(X,Y)$-pair, $X$ and $Y$ are two regions of the \emph{same} connected component of $\Sigma$. 

Each univalent vertex inherits a sign, which is \emph{positive}, resp. \emph{negative}, if the incident cut arc is locally oriented outwards, resp. inwards. 
Accordingly, an $(X,Y)$-pair as in rule 3) above, is called a \emph{positive} or \emph{negative} $(X,Y)$-pair, according to the sign of these vertices. We shall also simply say that these two $\circ$ dots are \emph{paired}. 
In figures,  we shall often replace the local orientation at univalent vertices with this sign, see the middle and the right-hand side of Figure \ref{fig:cond}. 

\begin{figure}
  \[ \begin{array}{ccc}
   \dessin{2cm}{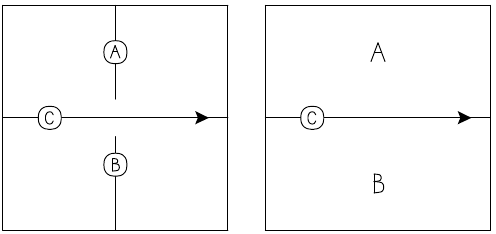} & \quad\dessin{2cm}{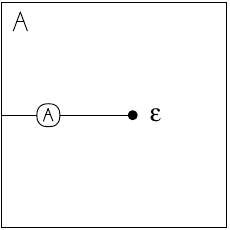}& \quad\dessin{2cm}{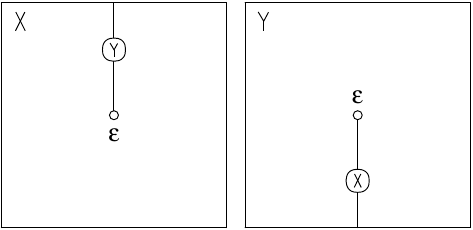}\\[0.5cm]
   \textrm{First rule} & \quad\textrm{Second rule} & \quad\textrm{Third rule}
  \end{array}\]
  \caption{Labeling rules for cut-diagrams (here $\varepsilon=\pm$)}
  \label{fig:cond}
\end{figure}

The key point of Definition \ref{def:cut} is that any surface-link map in $4$-space (in the sense of the introduction), given by a continuous map of $\Sigma$ into $\R^4$, yields a cut-diagram. 
This arises from the notion of \emph{broken surface diagrams}, which are the natural analogue of classical knot diagrams for surface-link maps, see e.g. \cite{CS,AMW}.  

Recall that broken surface diagrams correspond to generic immersions of surfaces into $\R^3$, obtained as a composition of a surface-link map and a projection from $\mathbb{R}^4$ to $\mathbb{R}^3$. 
This produces lines of transverse double points, which may contain singular points, meet at triple points and/or end at branch points. 
As for knot diagrams, double points are enhanced with an extra over/under information, which is encoded by cutting off a neighborhood of the lowest preimage. 
The resulting local models for triple points, branch points and singular points, are given in Figure \ref{fig:local}. 
\begin{figure}
   \[
\dessin{2.5cm}{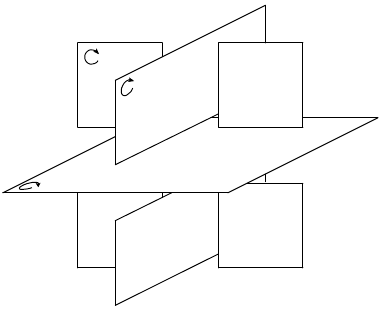}\quad \quad \quad \dessin{2.4cm}{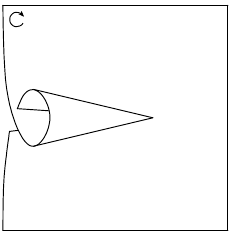}\quad  \quad \quad \quad \dessin{2.6cm}{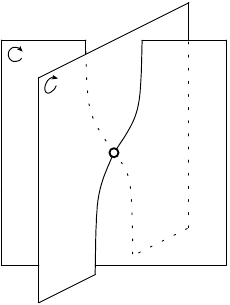}
\]
  \caption{Local models for triple points, branch points and singular points in a broken surface diagram}
  \label{fig:local}
\end{figure}
Observe that the over/under information \lq swaps\rq\, when traversing a singular point along a line of double points. 
Each line of double points also inherits a natural orientation from the ambient orientation and that of $\Sigma$.\footnote{More precisely, the orientation is chosen so that the local frame given by a positive normal vector to the overpassing region, a positive normal vector to the underpassing region, and a positive tangent vector to the line of double points, agrees with the ambient orientation of $\R^3$.}

An example is given on the left-hand side of Figure \ref{fig:yooohooo}, in the case of an embedding of $\Sigma=S^2$. 
\begin{figure}
\[
\dessin{3.6cm}{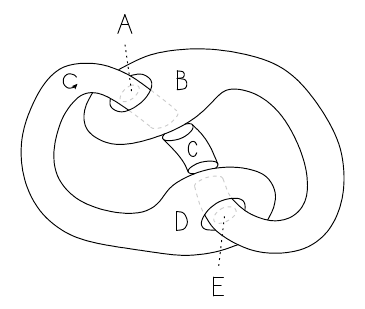}\ \leadsto\ \ \dessin{3.15cm}{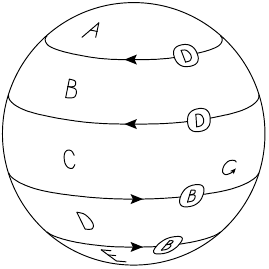}
\]
  \caption{Cut-diagram arising from a broken surface diagram}
  \label{fig:yooohooo}
\end{figure}

Now consider the abstract surface $\Sigma$, endowed with the preimages of all double points, and consider \emph{only the lower point set}. 
This forms a union $P$ of oriented immersed circles and/or intervals in $\Sigma$, which splits $\Sigma$ into regions as above; each triple point of the surface diagram provides an over/under information at the corresponding crossing of $P$, splitting $P$ into cut arcs. 
Labeling each cut arc, by the region containing the preimage with highest coordinate at the corresponding line of double points, satisfies automatically the labeling rules of Definition \ref{def:cut}, and thus provides a cut-diagram over $\Sigma$ for the given surface-link map. See Figure \ref{fig:yooohooo} for an example. 

\begin{defi}\label{def:cutop}
 Cut-diagrams arising in this way, from a surface-link map in $4$--space, are called \emph{topological cut-diagrams}.
\end{defi}

\begin{remark}
  Two nonequivalent surface-link maps may
  give rise to the same cut-diagram.
 Yet, as we shall recall below, they still retain all the data of the fundamental group that are needed to define our invariants.  
\end{remark}

\subsection{Reduced group of a cut-diagram}\label{sec:cut_group}

Let $\mathcal{C}$ be a cut-diagram over $\Sigma$. 

\begin{defi}
The \emph{group} of $\mathcal{C}$ is the group $G(\mathcal{C})$ generated by its regions, and with a relation $B^{-1}A^C$ for every pair of regions $(A,B)$ that are adjacent along a $C$-labeled cut arc as in Figure \ref{fig:cond} (left).  
An $i$th meridian is a region of the $i$th component of $\Sigma$, for some $i$, when regarded as a generator of  $G(\mathcal{C})$. 
\end{defi} 
Notice that a relation $[X,Y]$ is in particular inherited from each $(X,Y)$-pair of $\circ$ dots in $\mathcal{C}$. 

\begin{nota}
 Given a region $A$ of the cut-diagram $\C$, we shall sometimes denote by $[A]$ its associated meridian, which is an element of $G(\C)$. 
\end{nota}

If $\mathcal{C}$ is a topological cut-diagram of a surface-link map $f$, it is easily verified that 
there is an isomorphism from $G(\mathcal{C})$ to the fundamental group of the exterior of the image $L$ of $f$, sending each $i$th meridian of $\C$ to a topological $i$th meridian of $L$. 
Moreover in this case, the assignment of $[R]\in G(\C)$ to each region $R$ of $\C$, is a \emph{coloring} in the sense of \cite{CKS}.  

Notice that any two $i$th meridians in  $G(\mathcal{C})$ are always conjugate, for any $i$, 
so that $G(\mathcal{C})$ is normally generated by a choice of one meridian for each component, and we can consider the \emph{reduced group}  $\nR G(\mathcal{C})$ (see Definition \ref{def:red}).
The isomorphism class of this quotient does not depend on the choice of meridians, and a presentation for this group is given in \cite[Thm.~5.18]{AMY}.

Now, let $\gamma$ be a path on $\Sigma$. 
We may freely assume that $\gamma$ meets $\mathcal{C}$ transversally in a finite number of regular points. 
We associate an element $w_\gamma$ in $G(\mathcal{C})$ as follows. 
For the $k$th intersection point between $\gamma$ and $\mathcal{C}$ met when running $\gamma$ according to its orientation,  
denote by $A_k$ the label of the cut arc met at this point, and by $\varepsilon_k$ the local sign of this intersection point.
Then\footnote{Our notation here differs from that of 
Section 2.2 of \cite{AMY}, where this element is denoted by $\widetilde{w}_\gamma$, while the notation $w_\gamma$ is used for a normalized element that is no longer needed in the present paper. }
$$
w_{\gamma}:=A_1^{\varepsilon_1}\cdots
A_{|\gamma\cap\mathcal{C}|}^{\varepsilon_{|\gamma\cap\mathcal{C}|}}\in G(\mathcal{C}). $$
We shall need the following, which is an easy consequence of the labeling rules of Definition \ref{def:cut}, see \cite[Lem.~2.10]{AMY}.  

\begin{lemma}\label{lem:redlong}
If $\gamma$ and $\gamma'$ are two homotopic generic paths, rel. boundary, on
the $i$th component of $\Sigma$, for some $i$, 
then $w_\gamma=w_{\gamma'}$ in the quotient of $\nR G(\mathcal{C})$ by 
the normal subgroup $N_{(i)}$ generated by $i$th meridians. 
\end{lemma}

\subsection{Topological moves for cut-diagrams}\label{cut_topmoves}

It is well-known that two broken surface diagrams represent equivalent surface-links if and only if they differ by a sequence of the seven Roseman moves given in \cite{Roseman3}. Note that one of these moves is known to be generated by the other six, see \cite{Yashiro}. 
One can translate Roseman moves into the langage of cut-diagrams.
The resulting so-called \emph{topological} moves are summarized, up to reflection, in Figure \ref{fig:topomoves}. 

Given a topological move on a cut-diagram, 
we call \emph{supporting disk}, the disk(s) intersecting the cut-diagram as shown in Figure \ref{fig:topomoves}. 
There, we identify the non represented regions of the cut-diagrams before and after the move; represented regions that are not contained in the supporting disks, are also identified canonically unless otherwise specified, see Remark \ref{rem:38}.
  
\begin{figure} 
  \[
  \hspace{-.75cm}  \begin{array}{ccccc}
    \dessin{1.7cm}{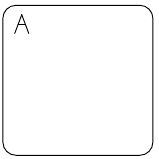} \stackrel{\textrm{T1}}{\longleftrightarrow} 
     \dessin{1.7cm}{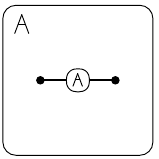} & & 
    \dessin{1.7cm}{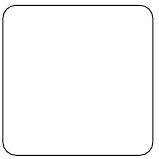} \stackrel{\textrm{T2}}{\longleftrightarrow} 
     \dessin{1.7cm}{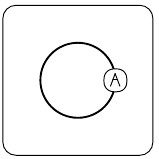} & & 
\dessin{1.7cm}{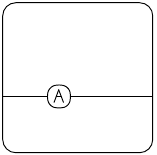}\stackrel{\textrm{T3}}{\longleftrightarrow} 
     \dessin{1.7cm}{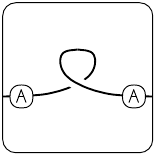}
  \end{array}
  \]
\vspace{.5cm}  
   \[
  \hspace{-.75cm}  \begin{array}{ccccc}
        \dessin{1.7cm}{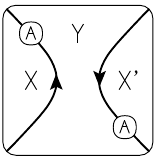} \stackrel{\textrm{T4}}{\longleftrightarrow} 
     \dessin{1.7cm}{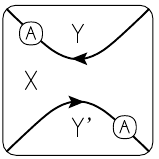} 
    & & 
    \dessin{1.7cm}{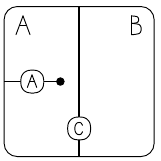} \stackrel{\textrm{T5}}{\longleftrightarrow} 
     \dessin{1.7cm}{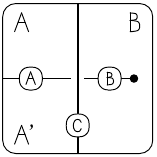} & & 
    \dessin{1.7cm}{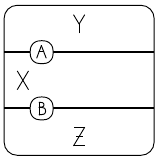} \stackrel{\textrm{T6}}{\longleftrightarrow} 
     \dessin{1.7cm}{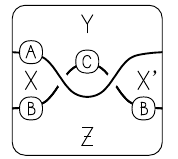} 
  \end{array}
  \]
  \vspace{.5cm}  
   \[
     \hspace{-.75cm}  \begin{array}{c}
     \dessin{1.7cm}{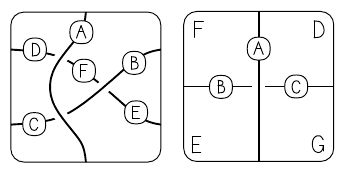} \stackrel{\textrm{T7}}{\longleftrightarrow} 
     \dessin{1.7cm}{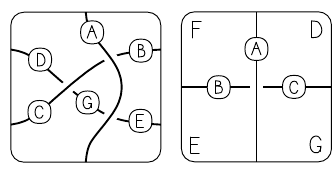} 
  \end{array}
  \]
  \caption[Topological moves for cut-diagrams]{Topological moves for cut-diagrams:\\ 
  \footnotesize{Here a move is valid if and only if, in the
    cut-diagrams before and after the move, the labeling rules of
    Definition \ref{def:cut} are fulfilled. In these pictures,
      different notations for regions or arc labels may refer to
      the same region. 
  }}
\label{fig:topomoves}
\end{figure}

\begin{remark}
We stress that, in order to be valid, some topological moves impose conditions on regions outside of the supporting disks. 
On one hand, each represented cut arc label $R$ implies the existence of a (possibly non represented) sheet supporting the region $R$ in the diagram. On the other hand, by the first labeling rules of Definition \ref{def:cut}, each crossing of cut arcs imposes the existence of some other sheet where the corresponding regions have the desired adjacency properties. 
Together with those in Figure \ref{fig:topomoves}, 
these non represented sheets are precisely the sheets involved in the corresponding Roseman move. 
Note that the case of move T7 is special, in the sense that the sheets supporting the \lq Reidemeister 3-like move\rq\, are not sufficient to ensure the existence of the other ones, which are therefore imposed as part of the move; this explains the fact that a T7 moves involves two supporting disks. 
\end{remark}

\begin{remark}\label{rem:38}
  Although topological moves of Figure \ref{fig:topomoves} are
  pleasingly reminiscent of usual local moves of knot theory, we
  stress that they are not entirely local moves on cut-diagrams, because 
  the labeling on cut arcs might change outside of the supporting disks.
  On one hand, a move that makes a region disappear, such as T2, T3 or T6, is only valid if this region never occurs anywhere as the label of some cut arc. 
  On the other hand, some of the topological moves, namely T4, T5 and T6, merge two regions, or split a region into two. In the former case, all the cut arcs labeled by one of the two merging regions are relabeled by the new region; in the latter case, the cut arcs labeled by the split region are relabeled by any of the two new regions, in such a way that the labeling rules of Definition \ref{def:cut} are satisfied. 
 We shall see in Lemma \ref{lem:iso} that these \lq labeling issues\rq\, are however harmless as far as the cut-diagram group is concerned.
\end{remark}

  By Roseman's theorem \cite[Thm.~1]{Roseman3}, we have the following.
\begin{theo}\label{th:m}
Two topological cut-diagrams of isotopic surface-links, are related by a sequence of topological moves.  
\end{theo}

\begin{remark}\label{rem:move_rules}
 Theorem \ref{th:m}  was already observed in \cite{AMY}. The list of moves given in \cite[Fig.~14]{AMY} is a different one, but it can easily be shown using the labeling rules of Definition \ref{def:cut}, to be equivalent to those of Figure \ref{fig:topomoves}. 
\end{remark}

\begin{lemma}\label{lem:iso}
 Suppose that $\C'$ is a cut-diagram obtained from $\C$ by some topological move.
 There is a canonical isomorphism from $G(\C)$ to $G(\C')$, which 
 yields the identity on the region labeling of any cut arc that is not included in the supporting disks. 
\end{lemma}
\begin{proof}
The canonical isomorphism from $G(\C)$ to $G(\C')$, is merely a reformulation and straightforward extension of a well-known fact on colorings of surface-links, see e.g. \cite[Thm.~3.4]{PR}. 
 This is verified by successively analyzing the seven moves of Figure \ref{fig:topomoves}, which naturally yield associated Tietze transformations between group presentations. 
 We note that this isomorphism is the identity on each generator $[R]$ such that region $R$ does not intersect the supporting disk;  a region $R$ of $\C$, resp. $\C'$, that gets split under a topological move T4, T5 or T6, gives rise to two new regions $R_1$ and $R_2$, but one easily verifies that the corresponding group elements $[R_1]$ and $[R_2]$ are equal in $G(\C')$, resp. $G(\C)$. 
 The latter part of the statement follows, since a region which is 
 included in the supporting disk of a topological move never occurs as a cut arc label. 
\end{proof}
 
  \subsection{Self-singular moves for cut-diagrams}\label{cut_singmoves}

Roseman moves were extended to broken surface diagrams of surface-link maps in \cite{AMW}, 
where three self-singular Roseman moves were introduced. 
The \emph{self-singular moves} shown in Figure \ref{fig:singmoves} 
are the translations of these self-singular Roseman moves in the langage of cut-diagrams. 
Here, we identify regions before and after each move in a canonical way, as with topological moves, and we define the supporting disks analogously. 
The first two lines of the figure corresponds to passing a singularity across a triple point: there are two versions, depending on whether the sheet that does not contain the singularity (region $C$ in the figure) passes under (move S1) or over (move S2) the other two. Move S3 of Figure \ref{fig:singmoves} is the cut-diagram version of a finger/Whitney move, while move S4 encompasses a cusp-homotopy, that passes a singularity across a branch point. 

\begin{figure}
 \[
 \hspace{-.75cm}  \begin{array}{c}
    \dessin{1.7cm}{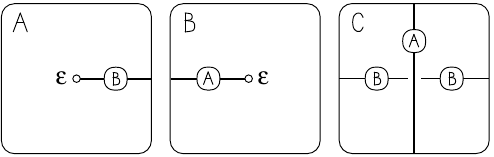} \stackrel{\textrm{S1}}{\longleftrightarrow} 
    \dessin{1.7cm}{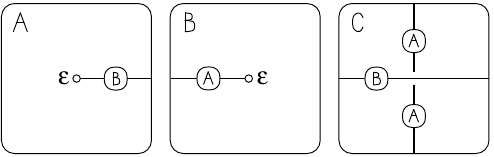}
   \end{array} \]
   \vspace{.25cm}  
\[\begin{array}{c}
    \dessin{1.7cm}{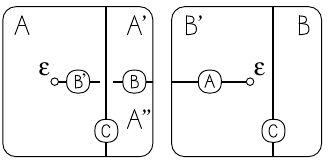} \stackrel{\textrm{S2}}{\longleftrightarrow} 
    \dessin{1.7cm}{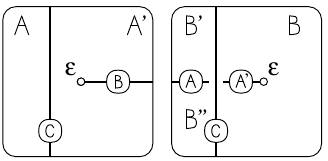}\qquad\,
   \end{array}
  \]
\vspace{.5cm}  
 \[
 \hspace{-.75cm}  \begin{array}{ccc}
    \dessin{1.7cm}{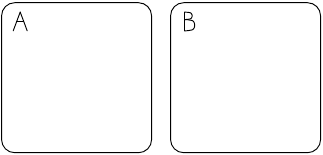} \stackrel{\textrm{S3}}{\longleftrightarrow} 
    \dessin{1.7cm}{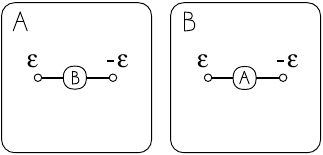}
 & \quad \quad & 
    \dessin{1.7cm}{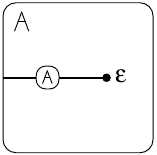} \stackrel{\textrm{S4}}{\longleftrightarrow} 
    \dessin{1.7cm}{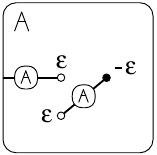} 
  \end{array}
  \]
  \caption{Self-singular moves for cut-diagrams:\\{\footnotesize
      here, $\varepsilon=\pm$, and $A$ and $B$ are always regions of a
      same connected component; in each move, $\circ$ dots decorated
      by the same sign are paired. In these pictures,
      different notations for regions or arc labels may refer to
      the same region. 
      }}\label{fig:singmoves}
\end{figure} 

As a consequence of \cite[Prop.~2.4]{AMW}, we have:
\begin{theo}\label{th:lh}
Two cut-diagrams of link-homotopic surface-link maps, are related by a sequence of topological or self-singular moves.  
\end{theo}
\begin{remark}
Note that the self-singular moves of Figure \ref{fig:singmoves} differ from those given in \cite[Fig.~16]{AMY}, but are generated by the latter ones. 
Another difference with \cite{AMY} lies in the definition of link-homotopy for surface-link maps (see \cite[Rem.~7.7]{AMY}); both notions however coincide for link maps.
\end{remark}

These topological and self-singular moves more generally define an equivalence relation on  cut-diagrams, which we shall call \emph{self-singular equivalence}. 
Note that  this equivalence relation encompasses the link-homotopy by Theorem \ref{th:lh}. 
\medskip

The following natural extension of Lemma \ref{lem:iso} to the singular setting was already observed in \cite{AMW} in the topological case, and is easily verified by analyzing each self-singular move. 
\begin{lemma}\label{lem:isoR}
 Suppose that $\C'$ is a cut-diagram obtained from $\C$ by some self-singular move.
 There is a canonical isomorphism from $\nR G(\C)$ to $\nR G(\C')$, which yields the identity on the region labeling of any cut arc that is not included in the supporting disks.  
\end{lemma}

\section{Higher order link maps invariants: proof of the invariance Theorems}\label{sec:proof_linkmaps}

The purpose of this section is to use the theory of cut-diagrams to prove the invariance Theorems \ref{th:main0}, \ref{th:main} and \ref{th:main2}. 
Therefore, throughout this section we restrict ourselves to the case where $\Sigma$ is a union of $n$ copies of $S^2$.

\subsection{The based case} 

Let us first prove the first half of Theorem \ref{th:main0} (the second half is shown at the end of Section \ref{sec:unbased}). 
This will come as a direct corollary of Theorem \ref{thm:cutmain0} below, which is a more general version for cut-diagrams. Hence we shall begin by extending our invariants to general cut-diagrams. 

Let $\mathcal{C}$ be a cut-diagram over $\Sigma$. 

\begin{defi}
 A \emph{basing} of $\CC$ is a choice of a region $R_i$ of the $i$th component $\Sigma_i$ of $\Sigma$, for each $i$. 
 This choice is materialized by a basepoint $b_i$ in the interior of $R_i$.   
 \end{defi}

Let us fix some index $i$. 
We denote by $\nR G(\C)_i$ the quotient of $\nR G(\mathcal{C})$ by 
the normal subgroup $N_{(i)}$ generated by $i$th meridians.\footnote{Note that $\nR G(\C)_i$ 
is the reduced group of the cut-diagram obtained from $\C$ by deleting the $i$th component, and deleting all cut arcs labeled by regions from the $i$th component. } 
As an immediate corollary to \cite[Thm.~5.18]{AMY}, we have that $\nR G(\C)_i$ is isomorphic to the reduced free group $\nR F^i_{n-1}$, where 
$F^i_{n-1}$  stands for the free group  on the $n-1$ generators $\{R_1,\cdots,R_n\}\setminus \{R_i\}$. 
We denote by $\phi: \nR G(\C)_i \rightarrow \nR F^i_{n-1}$  this  isomorphism. 

Following Definition \ref{def:positive}, we say that an element $x$ of $\nR G(\C)_i$ is \emph{positive}, 
if it is trivial or if the first non vanishing term of $E_i(\phi(x))$ has positive coefficient.

Let $(p_1,p_2)$ be two $\circ$ dots forming an $(X,Y)$ pair of $\mathcal{C}$, 
where $X$ and $Y$ are two regions of the $i$th component of $\Sigma$. 
Pick two paths $\alpha_1$ and $\alpha_2$, running  from $b_i$ to $p_1$ and $p_2$, respectively. 
We define an element $w_{p_1,p_2}$ of $\nR G(\C)_i$  as follows. 
$$ w_{p_1,p_2}:=\left\{\begin{array}{ll}
                            w_{\alpha_1}\cdot w_{\alpha_2}^{-1} & 
                            \textrm{ if $w_{\alpha_1}\cdot w_{\alpha_2}^{-1}$ is positive,}\\
                            w_{\alpha_2}\cdot w_{\alpha_1}^{-1} & \textrm{ otherwise.}
                            \end{array}
                            \right. $$

Set 
\begin{equation}\label{eq:cut-kirky}
 S_i(\C) := \sum_{\{p_1,p_2\}} \ve(p_1,p_2)(\phi(w_{p_1,p_2})-1)\in  \mathbb{Z}\nR  F^i_{n-1}, 
\end{equation}
\noindent where the sum runs over all subsets 
$\{p_1,p_2\}$ of the partition of $\circ$ dots on the $i$th component
of $\C$ (see Definition \ref{def:cut}), and where $\ve(p_1,p_2)$ is the common sign of such a pair, in the sense of Section \ref{sec:cut_def}. 

\begin{theo}\label{thm:cutmain0}
For any index $i$, 
$S_i(\C)$ is a self-singular equivalence invariant of the cut-diagram $\C$ endowed with a basing.  
\end{theo}

\begin{remark}\label{rem:i}
 If $\C$ is a topological diagram of some link map $L$, then the pair $(p_1,p_2)$ is given by the two preimages of a singular point $p$ of $L$, 
 and  $w_{p_1,p_2}$ is an associated \emph{positive} element of $\nR G(L_i)$, as defined in Section \ref{sec:pos}, so that we have $S_i(\C)=S_i(L)$. 
Hence Theorem \ref{thm:cutmain0}, combined with Theorem \ref{th:lh}, readily implies the first half of Theorem \ref{th:main0}. 
\end{remark}

\begin{proof}
By definition, it suffices to show that $S_i(\C)$ remains  unchanged under a topological or self-singular move. 

The case of topological moves is a routine verification. 
More precisely, for all topological moves we may freely assume up to homotopy that $\alpha_1\cup \alpha_2$ is disjoint from the supporting disks. 
By Lemmas \ref{lem:redlong} and \ref{lem:iso}, we then have that $w_{p_1,p_2}$ remains unchanged in the (reduced) cut-diagram group for any pair $(p_1,p_2)$ of $\circ$ dots. 

For a self-singular move, we similarly observe using Lemmas \ref{lem:redlong} and \ref{lem:isoR} that $E_i(w_{p_1,p_2})$ does not change, 
for any pair $(p_1,p_2)$ of $\circ$ dots that is disjoint from the supporting disks. 
But we must also consider the situation where the pair $(p_1,p_2)$ is contained in the supporting disks of the move. 
This is a case-by-case verification as follows: 
\begin{itemize}
   \item Move $S1$ does not affect a neighborhood of the $\circ$ dots that it involves, hence leaves the associated paths locally unchanged.
   \item Move $S2$ is more delicate, since it changes these paths in the supporting disks. 
   More precisely, suppose that a move $S2$ involves the pair $(p_1,p_2)$ of $\circ$ dots on the $i$th component of $\Sigma$, 
   and that the paths $\alpha_1$ and $\alpha_2$ are as shown on the left-hand side of Figure \ref{fig:zemove}. 
   The associated words are of the form $w_{\alpha_1}=g$ and $w_{\alpha_2}=hD^{\ve}$ for some $g,h\in G(\C)$ and some sign $\ve$. 
   Denote by $\alpha'_1$ and $\alpha'_2$ the images of $\alpha_1$ and $\alpha_2$, respectively, under this move $S2$, see  the right-hand side of Figure \ref{fig:zemove}.  
   Then we have  $w_{\alpha'_1}=D^{-\ve}g$ and $w_{\alpha'_2}=h$, and  the word $w_{p_1,p_2}$ itself remains unchanged. 
   \begin{figure}
   \[
\hspace{-.75cm}  \begin{array}{ccc}
    \dessin{2cm}{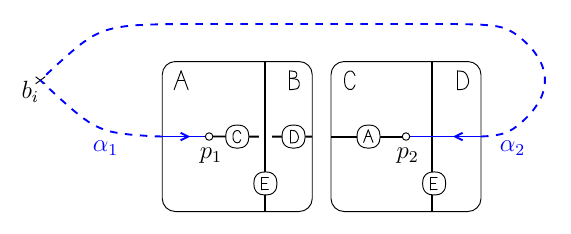} & \leftrightarrow &\dessin{2cm}{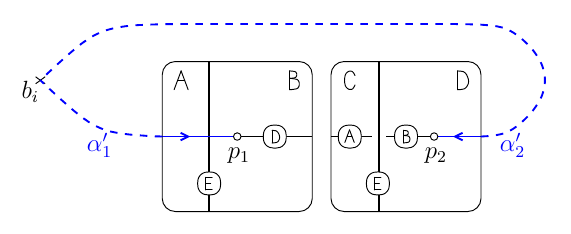} 
  \end{array}
  \]
  \caption{Applying an $S2$ move}\label{fig:zemove}
  \end{figure}
  \item A move $S3$ applied from left to right in Figure \ref{fig:singmoves}, introduces two pairs of $\circ$ dots with opposite signs, 
  hence two extra terms in $S_i(\C)$. But by Lemma \ref{lem:redlong}, the paths associated with these two pairs of $\circ$ dots may be freely chosen to be 
  parallel, so that they represent the same (reduced) group element. The two extra terms thus cancel in $S_i(\C)$. 
  \item A move $S4$ introduces a pair of $\circ$ dots $(p_1,p_2)$ in its supporting disk, hence one extra term in $S_i(\C)$. 
    By Lemma \ref{lem:redlong}, the associated paths $\alpha_1$ and
    $\alpha_2$ to the basepoint $b_i$ can be chosen to be parallel
    outside of this disk,
   so that they represent the same group element.
    It follows that $w_{p_1,p_2}=1$, so that  this extra pair  does not contribute to $S_i(\C)$.
\end{itemize}
This concludes the proof that $S_i$ is a self-singular equivalence invariant of based cut-diagrams.  
\end{proof}

\subsection{The unbased case}\label{sec:unbased}

In order to prove Theorems \ref{th:main} and \ref{th:main2}, 
we shall likewise consider cut-diagram versions of these constructions. 
We shall again make use of Notation \ref{nota:coeff} throughout.

On one hand, following Section \ref{sec:kappa}, we take the reduced Magnus expansion $E_i$ of $S_i(\C)$, for each $i$. 
We obtain in this way finite sums $E_i(\C) := E_i\left( S_i(\C) \right)$, similar to Equation (\ref{eq:Ekirky}):
  $$ E_i(\C) = \sum_{I} \kappa_{E_i(\C)}(I)X_{I}\in   \Lambda_i .$$
For any non-repeated sequences $I$ in $\{1,\ldots,n\}\setminus\{i\}$, we denote by 
$\widetilde{\kappa}_{\C}(I;i)$ the residue class of $\kappa_{E_i(\C)}(I)$ modulo $D_{E_i(\C)}(I)$. 

On the other hand, 
following Section \ref{sec:refine}, one can rewrite Equation (\ref{eq:cut-kirky}) as 
$$ S_i(L) = \sum_{g\in \nR G(\mathcal{C})_i^+} \rho_\C(g)(\phi(g) -1)\in \mathbb{Z}\nR  F^i_{n-1} $$
for some integer coefficients $\rho_\C(g)$, 
where  $\nR G(\mathcal{C})_i^+$ is the set of positive elements in $\nR G(\C)_i$. 
For any $g\in \nR G(\C)_i^+$ and any non-repeated sequence $I$ in $\{1,\cdots,n\}\setminus\{i\}$, 
denote by 
$\tilde{\kappa}_\C(I;g)$ the residue class of $\kappa_{E_i(\phi(g))}(I)$ modulo $D_{E_i(\phi(g))}(I)$. 

Theorems \ref{th:main} and \ref{th:main2} follow immediately from the next result, by the same observation as in Remark \ref{rem:i}. 

\begin{theo}\label{thm:cutmain}
Let $i\in \{1,\cdots,n\}$, and let $I$ be a non-repeated sequence in $\{1,\cdots,n\}\setminus \{i\}$. 
\begin{enumerate}
\item 
The residue class $\widetilde{\kappa}_{\C}(I;i)$ is a self-singular equivalence invariant of the cut-diagram $\C$. 
 \item
 $\textrm{ }$\\[-0.87cm]
$$ \mathcal{K}_{i}(\C):=\left\{ \left(\rho_\C(g), \small{\sum}_{\large{I}} \tilde{\kappa}_\C(I;g)X_I \right)\, \vert \, \textrm{$g\in \nR G(\C)_i^+$ such that $\rho_\C(g)\neq 0$} \right\} $$
is a self-singular equivalence invariant of the cut-diagram $\C$. 
\end{enumerate}
\end{theo}

The rest of this section is devoted to the proof of Theorem \ref{thm:cutmain}, and we assume that an index $i\in \{1,\cdots,n\}$, and a non-repeated sequence  $I$ in $\{1,\cdots,n\}\setminus \{i\}$ is fixed throughout. 
(We shall also prove the second half of Theorem \ref{th:main0} at the end of the section.)
\\
 
Since these invariants are all uniquely determined by $S_i(\C)$, which is a self-singular equivalence invariant of based cut-diagrams by Theorem \ref{thm:cutmain0}, it suffices to show the independence under basing change. 
We first note that the exact same argument as in Remark \ref{rem:positive}, ensures that the notion of positivity  in $\nR G(\C)_i$, is independent of the choice of basing for $\mathcal{C}$. 
 A different choice of basing affects the construction in two ways, as discussed below.\\

 On one hand, this changes the paths $\alpha_1$ and $\alpha_2$ associated to any pair of $\circ$ dots. 
 Changing the basepoint $b_i$ to $b'_i$ on the $i$th component of $\C$, indeed changes the pair of arcs $(\alpha_1,\alpha_2)$ to 
 $(\gamma\alpha_1,\gamma\alpha_2)$ for some path $\gamma$ running from $b'_i$ to $b_i$.  
 Note that the isomorphism $\phi:\nR G(\C)_i\rightarrow \nR F^i_{n-1}$ is not modified by this basing change.
 In the reduced free group $\nR F^i_{n-1}$, this  
turns the element $w_{p_1,p_2}$ into its conjugate $w_\gamma w_{p_1,p_2} w_\gamma^{-1}$. 
\begin{remark}\label{rem:basemoi}
The above argument in particular implies that changing the basing of the cut-diagram $\C$, changes $S_i(\C)$ by a conjugate. 
\end{remark}
Now let $W$ be an element of  $\Lambda_i$. 
Setting $E_i(w_\gamma) = 1+U\in \Lambda_i$, we have 
$E_i(w_\gamma) W E_i(w_\gamma^{-1}) = W +UW - WU$, 
and therefore, for any sequence $I$ in $\{1,\cdots,n\}$ we have 
$$ \kappa_{E_i(w_\gamma) W E_i(w_\gamma^{-1})}(I) = \kappa_{W}(I) + \sum_{I=I_1I_2} \left( \kappa_U(I_1)\kappa_W(I_2) - \kappa_W(I_1)\kappa_U(I_2) \right), $$
\noindent where the sum runs over all nonempty sequences $I_1,I_2$ whose concatenation $I_1I_2$ is the sequence $I$. 
We thus obtain the following general observation: 
\begin{claim}\label{claim0}
 We have $D_{E_i(w_\gamma)WE_i(w_\gamma^{-1})}(I)=D_{W}(I)$, and 
 $\kappa_{E_i(w_\gamma) W E_i(w_\gamma^{-1}))}(I)\equiv \kappa_{W}(I)\textrm{ mod $D_{W}(I)$}$.  
\end{claim}
In particular, taking respectively $W=E_i(\C)$ and $W=E_i(\phi(g))$ for all $g\in \nR G(\C)_i^+$, ensures the desired invariance property for (1) and (2). \\
 
 On the other hand, changing the basing at $b_j$ ($j\neq i$), modifies the isomorphism $\phi$ from the quotient $\nR G(\C)_i$   to the reduced free group $\nR F^i_{n-1}$. 
Suppose that we change the basing on the $j$th component  of $\mathcal{C}$, for some $j$.  
As observed in Remark \ref{rem:positive}, this implies substituting, in the reduced free group, the $j$th generator $x_j$ by a conjugate $g^{-1}x_jg$ from some $g\in \nR F^i_{n-1}$. 
At the level of the reduced Magnus expansion, this implies substituting 
each occurence of $X_j$ by $X_j+X_j U-U X_j$ for some $U \in \Lambda_i $.

Let us again consider an element $W$ of  $\Lambda_i $. We shall need the following: 
\begin{claim}\label{claim2}
Let $j\in \{1,\cdots ,n\}$. 
Then $D_W(I)$ and the residue class $\kappa_W(I)$ mod $D_W(I)$ are invariant under substituting  each occurence of the variable $X_j$ by $X_j+X_j U - U X_j$ for some $U \in \Lambda_i $. 
\end{claim}
Assuming this result, the proof of (1) follows immediately by taking $W=E_i(\C)$ in Claim \ref{claim2}.
The proof of (2) likewise uses Claim \ref{claim2} with $W=E_i(\phi(g))$ for all $g\in \nR G(\C)_i^+$, but requires one additional argument. 
Changing the isomorphism $\phi$, may indeed change the fact that a given positive element in $\nR G(\C)_i^+$ contributes to 
$\mathcal{K}_{i}(\C)$. 
But if $\phi': \nR G(\C)_i \rightarrow \nR F^i_{n-1}$ is another isomorphism, induced by a different choice of basing, then there is a one-to-one correspondence between the sets $\{ g\in \nR G(\C)_i^+\, \vert \, \phi(g)\neq 1 \}$ and $\{g\in \nR G(\C)_i^+ \, \vert \, \phi'(g)\neq 1 \}$, which induces a bijection between the two corresponding sets $\mathcal{K}_{i}(\C)$.

It thus only remains to prove Claim \ref{claim2} to complete the proof of Theorem \ref{thm:cutmain}.
\begin{proof}[Proof of Claim \ref{claim2}]
The proof of this claim is by induction on the length $k$ of the sequence $I$. 
The case $k=1$ is trivial, since the substitution leaves the degree $1$ part of $W$ unchanged.  
The induction hypothesis ensures that, for a length $k$ sequence $I$, $D_W(I)$ is indeed invariant. 
Now, substituting each occurence of $X_j$ by $X_j+X_j U - U X_j$ in $W$, may create new terms that contribute to the coefficient $\kappa_W(I)$. 
But the indeterminacy $D_W(I)$ was precisely chosen so that all  such extra terms, which arise from lower degree terms in $W$, vanish modulo $D_W(I)$.\footnote{Note that a very similar argument already appears in \cite[Proof of (13)]{Milnor2}.}
Therefore $\kappa_W(I)$ mod $D_W(I)$ is invariant under substitution, which proves the claim. 
\end{proof}

\begin{remark}\label{rem:cor}
Observe that the above proof actually shows that, for each index $i$, and for any non-repeated sequence $I$ in $\{1,\ldots,n\}\setminus\{i\}$,  the multiset 
 $$ \left\{ \left(\rho_\C(g), \widetilde{\kappa}_\C(I;g) \right)\, \vert \, \textrm{$g\in \nR G(\C)_i^+$ ; $\rho_L(g)\neq 0$} \right\} $$
 is a self-singular equivalence invariant for $\C$. This fact readily implies Corollary \ref{th:main3}. 
\end{remark}

We finally prove the second half of Theorem \ref{th:main0}. 
This is a direct consequence of Lemma \ref{lem:fini} below, following Remark \ref{rem:i}. 
In what follows, we shall call \emph{trivial} the cut-diagram with no cut-arc.

\begin{lemma}\label{lem:fini}
If $\C$ is self-singular equivalent to the trivial 
(unbased) cut-diagram, then $S_i(\C)=0$ for all $i$ and for any choice of basing.
\end{lemma}
\begin{proof}
Suppose that $\C$ is self-singular equivalent to the trivial
(unbased) cut-diagram: this means that $C$ can be deformed into the trivial cut-diagram by a finite sequence of topological or self-singular moves. 
Now pick any basing on $\C$. Up to some basing change, we may assume that the above sequence of moves can still be performed, resulting in the (based) trivial cut-diagram $U$, which clearly satisfies $S_i(U)=0$. 
As observed in Remark \ref{rem:basemoi}, changing the basing of a cut-diagram, only changes the sum $S_i$ by a conjugate.  This shows that $S_i(\C)=0$ for the given basing, and concludes the proof.
\end{proof}

\section{Examples}\label{sec:ex}

\subsection{The Fenn-Rolfsen link map}

As a warmup, let us begin with recalling in some details the construction of the \emph{Fenn-Rolfsen link map} $FR$, which is the first example of a link map that is not link-homotopically trivial \cite{FR}. 
$FR=K_1\cup K_2$ is a 2-component link map, each sphere having a single (self-)singularity. 
It is built as follows, using the cross-section picture given in Figure \ref{fig:FR}: 
the figure depicts successive \lq slices\rq\, of $FR$, which are its intersection with $\R^3\times \{t\}\subset \mathbb{R}^4$ for various values of $t$.  
Consider in $\R^3\times \{0\}$ the Whitehead link $W$. 
A single self-crossing change on the first component of $W$, turns $W$ into the unlink. 
In $\R^3\times \{t\}$ for $t>0$, we consider the trace of this deformation, and cap off by two disks. 
More precisely, we deform the first component so that a self-intersection, represented by a $\circ$ in Figure \ref{fig:FR}, is created at $t=1$, thus producing at $t=2$ the result of a (self-)crossing change on the first component of $W$; an isotopy of the resulting link then yields the trivial link at $t=3$, which we then cap off by two disks at $t>3$. 
In $\R^3\times \{t\}$ for $t<0$, owing to the symmetry of the Whitehead link $W$, we may apply the same construction on component $2$. 
\begin{figure}
 \includegraphics[scale=0.65]{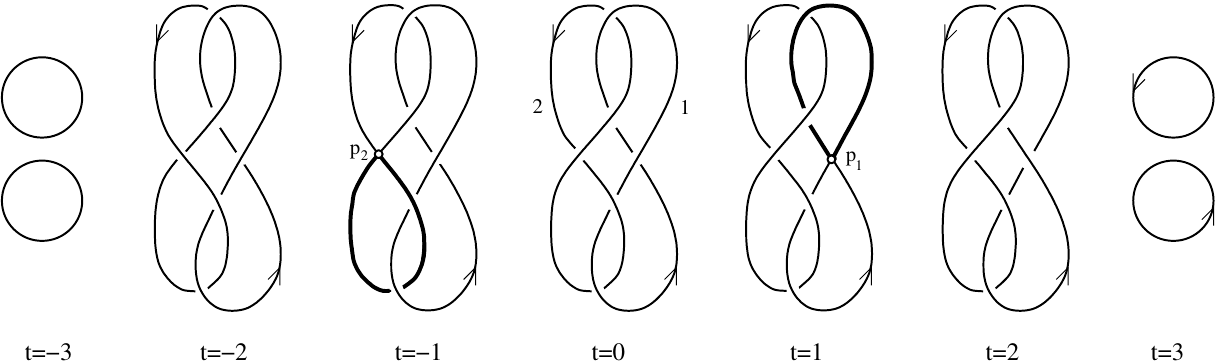}
 \caption{The Fenn-Rolfsen link map}\label{fig:FR}
\end{figure}

This process is sometimes called a \emph{Jin-Kirk suspension} of the Whitehead link $W$ in the literature. This construction applies more generally to any $n$-component link $K_1\cup\cdots \cup K_n$ such that $K_i$ is null-homotopic in the complement of $L\setminus K_i$ for two distinct values of $i$ in $\{1,\cdots,n\}$ (see the following three  subsections).  

Now, the invariant detecting $FR$ up to link-homotopy is essentially the Kirk invariant. This is a straightforward computation, along the following lines. 
Denote respectively by $p_1$ and $p_2$ the singularities on the first and second components of FR; note that $\ve(p_2)=-\ve(p_1)=1$.  
For $i=1,2$, the loop $\alpha_{p_i}$ on the $i$th component $K_i$ based at $p_1$, may freely be chosen to lie on the same slice as $p_i$, as represented in bold in Figure \ref{fig:FR}. 
We thus get immediately that the Kirk invariant satisfies 
$$ \left(\sigma_1(FR),\sigma_2(FR)\right) =  \left(1-t,t-1\right). $$

\subsection{Realizing higher order Kirk invariants}\label{sec:real}

 For $n\ge 3$, consider the link map $Y[n]$ in $\R^4$ given by Jin-Kirk suspension on the $n$-component link depicted in the middle of Figure \ref{fig:Link}.
\begin{figure}
 \includegraphics[scale=0.75]{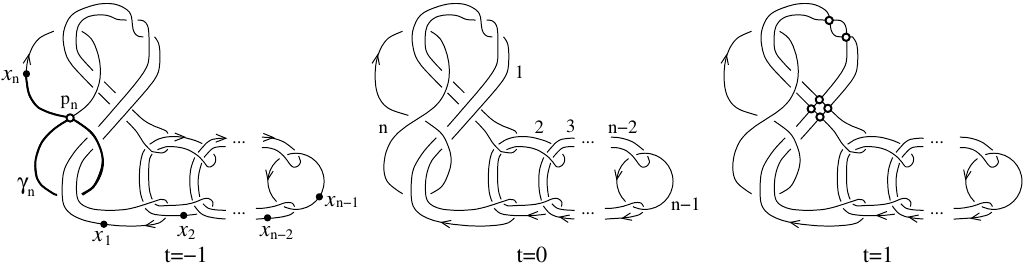}
 \caption{The link map $Y[n]$}\label{fig:Link}
\end{figure}
More precisely, $Y[n]$ is obtained as follows. 
In $\R^3\times \{t\}$ for $t<0$, consider the trace of a link-homotopy on the $n$th component: a singular point $p_n$ is created at (say) $t=-1$, and the resulting link living in a slice at $t<-1$ is isotopic to the trivial $n$-component link so that it can be capped off by disks.
In $\R^3\times \{t\}$ for $t>0$, consider the trace of a link-homotopy on the first component, creating six singularities that we assume to all lie in the slice at $t=1$; we leave it to the reader to verify that the resulting link in a slice at $t>1$ is also isotopically trivial, thus completing the construction.

\begin{lemma}\label{prop:real}
 For any value of $n\ge 3$, we have $\kappa_{Y[n]}\left(I;n\right) =0$ for any non-repeated sequence of at most $n-2$ indices in $\{1,\cdots,n-1\}$, and 
$$\tilde{\kappa}_{Y[n]}\left(1\cdots n-1;n\right) = 1. $$ 
Moreover,  $\mathcal{K}_{Y[n]}(I;n)$ is trivial for any non-repeated sequence of at most $n-2$ indices in $\{1,\cdots,n-1\}$, and  $\mathcal{K}_{Y[n]}(1\cdots n-1;n)=\left\{(1,1)\right\}$.\\
In particular, the link map $Y[n]$ is not link-homotopically trivial. 
\end{lemma}
\begin{proof}
Note that $\varepsilon(p_n)=+1$.  
Pick as basepoint on the $n$th component, the point $x_n\in \R^3\times\{-1\}$ shown on the left-hand side of Figure \ref{fig:Link}, and denote by $\gamma_n$ the loop based at $x_n$ shown in bold in the figure. 
 Pick further basepoints $x_i$ for other components of $Y[n]$, $i=1,\cdots,n-1$, as shown in the figure, all also lying in the slice $\R^3\times\{-1\}$. 
 By Lemma \ref{lem:isor}, this choice of basing for $Y[n]$ induces an isomorphism 
 $\nR G(\tilde{Y})\simeq \nR F^n_{n-1}$, 
 where $\tilde{Y}$ denotes the $(n-1)$-component link map obtained from $Y[n]$ by deleting the $n$th component, and where $F^n_{n-1}$ is the free group generated by $x_1,\cdots,x_{n-1}$. 
 Next consider the $(n-1)$-component link 
 $L:= \tilde{Y}\cap (\R^3\times \{-1\})$. Notice that the link $L$ is isotopic to the trivial link, and our choice of basing naturally yields isomorphisms 
 $$ \nR G(\tilde{Y})\simeq \nR G(L)\simeq \nR F^n_{n-1}, $$
 where $G(L)$ denotes the fundamental group of $(\R^3\times \{-1\})\setminus L$. 
Using the above isomorphisms and the link diagram in $\R^3\times \{-1\}$, one can easily express $g_{n}$, the unique positive element of $\nR G(\tilde{Y})$ representing the loop $\gamma_n$,  as
$$ g_{n} = \left[ x_1, \left[ x_2, \cdots  \left[x_{n-3},[x_{n-2},x_{n-1}]\right]\cdots \right] \right]\in \nR F^n_{n-1}. $$
\noindent (Notice that $L\cup \gamma_n$ is in fact a copy of Milnor's link, see \cite[\S 5]{Milnor1}.)
It follows that $$S_n(Y[n]) = (g_n-1).$$ 
The result follows, noting that the reduced Magnus expansion $E_n(g_n-1)$  is an alternate sum of $2^{n-2}$ monomials of degree $n-1$, starting with $X_1X_2\cdots X_{n-1}$.   
\end{proof}

\begin{remark}\label{rem:eko}
The family of link maps $Y[n]$, is a typical example of link maps that are detected by our higher order link-homotopy invariants, but that the variant suggested in Remark \ref{rem:eco}, which naturally arises from the covering spaces approach to the Kirk invariant, fails to detect. 
Using the notation of Section \ref{sec:covering}, we indeed have that 
$\sigma_n(Y[n]) = (g_n-1)(g_n^{-1}-1)$. 
The fact that $E_n(g_n-1)$  is a sum of monomials of degree $n-1$, implies that $E_n(g_n^{-1}-1)=-E_n(g_n-1)$, so that we have $E_n(\sigma_n(Y[n])) = 0$, for any value of $n\ge 3$. 
As a matter of fact, the same phenomenon occurs with the family of link maps $\mathcal{S}[n]$ introduced in the next subsection.
\end{remark}

\subsection{The (generalized) Stirling link map}\label{sec:stirling}

We next compute our link-homotopy invariants on another $1$-parameter family of link maps, which contains in the $3$-component case the example provided by Stirling in \cite[\S~5.3]{stirling23} (see also \cite[\S~6.2]{Stirling}) as  main computational example for his construction. 

Consider in $\mathbb{R}^3\times \{0\}$ the $n$-component link $L_S$ shown in the middle of Figure \ref{fig:stirlink}. 
 \begin{figure}
  \includegraphics[scale=0.75]{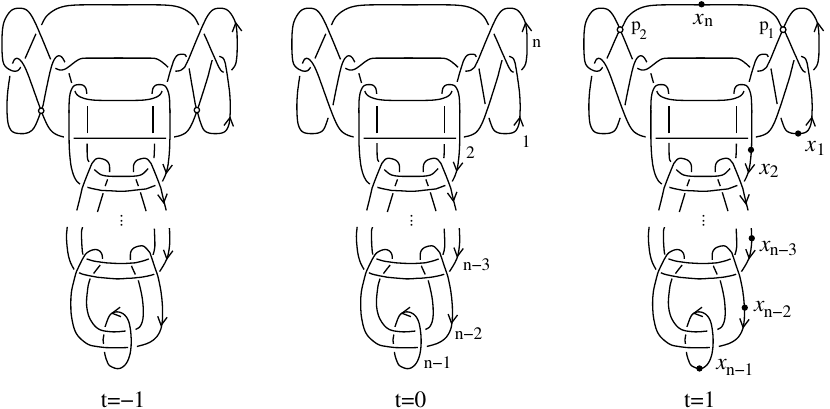}
   \caption{The (generalized) Stirling link map $\mathcal{S}[n]$}\label{fig:stirlink}
\end{figure} 

We define an $n$-component link map $S[n]$ by Jin-Kirk suspension on $L_S$ as follows. 
Consider the trace of a link-homotopy on the $n$th component of $L_S$, introducing two singularities $p_1$ and $p_2$ in  $\mathbb{R}^3\times \{1\}$. 
The resulting link in, say, $\mathbb{R}^3\times \{2\}$, is isotopic to the trivial link and can be capped off by disks in $\mathbb{R}^3\times \{t\}$ for some $t>2$. 
Similarly, we consider in $\mathbb{R}^3\times \mathbb{R}_{\le 0}$ the trace of a link-homotopy on the first component, that introduces two singularities in $\mathbb{R}^3\times \{- 1\}$ as shown in Figure \ref{fig:stirlink}. 

\begin{lemma}\label{lem:stirlink}
 For any value of $n\ge 3$, we have $\kappa_{S[n]}\left(I;n\right) =0$ for any non-repeated sequence of at most $n-2$ indices in $\{1,\cdots,n-1\}$, and 
$$\tilde{\kappa}_{\mathcal{S}[n]}\left(1\cdots n-1;n\right) = -1. $$ 
\noindent Moreover, $\mathcal{K}_{\mathcal{S}[n]}(1;n)=\{(1,1);(-1,1)\}$.
In particular, the link map $\mathcal{S}[n]$ is not link-homo\-to\-pi\-cally trivial.
\end{lemma}

\begin{remark}
Lemma \ref{lem:stirlink} also implies that $S[n]$ is not link-homotopic to the link map $Y[n]$ of Lemma \ref{prop:real}, for any $n\ge 3$.  
\end{remark}

\begin{proof}
We have $\varepsilon(p_1)=-\varepsilon(p_2)=+1$.  
Pick basepoints  $x_i\in \R^3\times\{1\}$ ($1=1,\cdots,n$) as shown on the right-hand side of Figure \ref{fig:stirlink}, and consider the loops $\gamma_1$ and $\gamma_2$ represented in Figure \ref{fig:stirlink2}, based at $x_n$ and passing through $p_1$ and $p_2$, respectively, while changing branches.  
 \begin{figure}
  \includegraphics[scale=0.75]{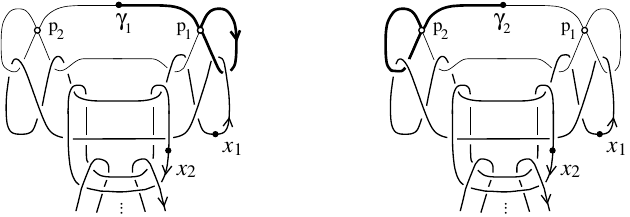}
   \caption{The based loops $\gamma_1$ and $\gamma_2$ on the $n$th component of $\mathcal{S}[n]$}\label{fig:stirlink2}
\end{figure} 
Let $\tilde{S}$ denote the $(n-1)$-component link map obtained from $\mathcal{S}[n]$ by deleting the $n$th component; 
 by Lemma \ref{lem:isor} we have $\nR G(\tilde{S})\simeq \nR F^n_{n-1}$, 
where $F^n_{n-1}$ is the free group generated by $x_1,\cdots,x_{n-1}$. 
The $(n-1)$-component link  $L:= \tilde{S}\cap (\R^3\times \{-1\})$ is isotopic to the trivial link, and we thus have 
$\nR G(\tilde{S})\simeq \nR G(L)\simeq \nR F^n_{n-1}$, 
 where $G(L)$ denotes the fundamental group of $(\R^3\times \{-1\})\setminus L$.
As in the proof of Lemma \ref{prop:real}, we can thus express the unique positive element $g_i\in \nR G(\tilde{S})$ representing the loop $\gamma_i$, for $i=1,2$ as
$$ g_{1} =x_{1}\,\, \textrm{ and } \,\, g_{2} = x_{1}^c\textrm{, where }c:=
\left[x_{2},\left[x_{3}, \cdots \left[x_{n-3},[x_{n-2},x_{n-1}]\right]\cdots \right]\right].$$
It follows that $$S_n(\mathcal{S}[n]) = (x_{1}-1) - (x_{1}^c-1).$$ 
The computation for the invariant $\tilde{\kappa}\left(1\cdots n-1;n\right)$ follows from the fact that 
$$ E_n(c)= 1 + X_2X_3\cdots X_{n-2} + \textrm{other terms of degree $n-2$}. $$
The two positive elements contributing to $S_n(\mathcal{S}[n])$ are $x_{1}$ and $x_{1}^c$, with coefficient $+1$ and $-1$, respectively. 
Since $E_n(x_{1})=1+X_{1}$ and 
$E_n(x_{1}^c)=1+X_{1}-(E_n(c)-1)X_{1} + X_{1}(E_n(c)-1)$, we 
obtain the desired values for $\mathcal{K}_{\mathcal{S}[n]}\left(1;n\right)$. 
\end{proof}

\subsection{Comparing the invariants}\label{sec:compare}

In this section, we provide examples showing that the
  link-homotopy invariants defined in Sections \ref{sec:refine} and \ref{sec:kappa}, 
  both detect link maps that the others do not. 

Let us first consider the $n$-component link map $\mathcal{S}^i[n]$, obtained from the link map  $\mathcal{S}[n]$ of Section \ref{sec:stirling} 
by reversing the orientation of the $i$th component, for any $i$ such that $1<i<n$.

\begin{lemma}\label{lem:compare}
The invariant $\widetilde{\kappa}(1\cdots n-1;n)$ distinguishes the two link maps $\mathcal{S}^i[n]$ and  $\mathcal{S}[n]$ up to link-homotopy. 
However, none of the invariants $\mathcal{K}_n$ or $\mathcal{K}(I;n)$ for any non-repeated sequence $I$ in $\{1,\cdots,n-1\}$, distinguishes these two link maps. 
\end{lemma}

\begin{proof}
We only have to show how changing the orientation on component $i$, affects the computations in the proof of Lemma \ref{lem:stirlink}; we shall freely use the notation used in this latter proof.
Let $\tilde{S}^i$ denote the $(n-1)$-component link map obtained from $\mathcal{S}^i[n]$ by deleting the $i$th component, so that  $\nR G(\tilde{S}^i)\simeq \nR F^i_{n-1}$.
For $i=1,2$, the positive element $g'_i\in \nR G(\tilde{S}^i)$ representing the loop $\gamma_i$ of Figure \ref{fig:stirlink2} (with the orientation on the $i$th component reversed), are given by 
 $$g'_1 =x_{1}\, \textrm{ and }\,g'_{2} = x_{1}^{c'}\textrm{, where }c':=
\left[x_{2},\left[ \cdots ,\left[x_{i}^{-1}, \cdots ,[x_{n-2},x_{n-1}]\right]\cdots \right]\right].$$
It follows that $S_n(\mathcal{S}^i[n]) = (x_{1}-1) - (x_{1}^{c'}-1)$.

Since $ E_n(c')= 1 - X_2X_3\cdots X_{n-2} + \textrm{other terms of degree $n-2$}$, 
we have 
$$ \widetilde{\kappa}_{\mathcal{S}^i[n]}(12\cdots n-1;n) = 1 = -\widetilde{\kappa}_{\mathcal{S}[n]}(12\cdots n-1;n). $$
On the other hand, we get 
$$ \mathcal{K}_3(\mathcal{S}^i[n]) = \left\{ (1,X_1);(-1,X_1) \right\} = \mathcal{K}_3(\mathcal{S}[n]),$$
hence 
$\mathcal{K}_{\mathcal{S}^i[n]}(I;n)=\mathcal{K}_{\mathcal{S}[n]}(I;n)$ for any non-repeated sequence $I$ in $\{1,\cdots,n-1\}$
\end{proof}

Next, let us consider a slightly modified version of the $3$-component link map $Y[3]$ used in Section \ref{sec:real}. 

Let $Y:=K_1\cup K_2\cup K_3$ be the link map in $\R^4$ obtained by Jin-Kirk suspension on the $3$-component link in $\R\times \{0\}$ depicted in Figure \ref{fig:link3}. 
\begin{figure}
 \includegraphics[scale=0.59]{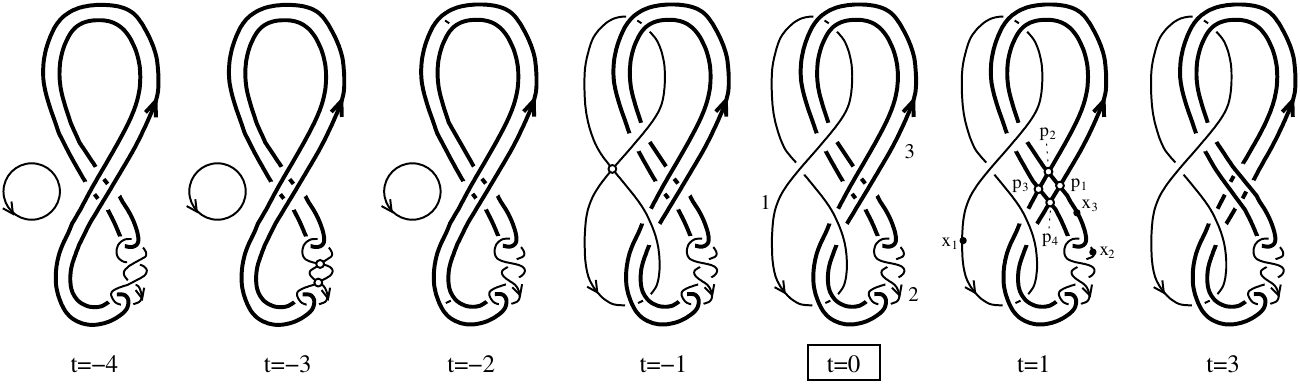}
 \caption{The $3$-component link map $Y$}\label{fig:link3}
\end{figure}
More precisely, take the trace of a link-homotopy on component $3$, that creates $4$ singularities, all lying in $\R^3\times \{1\}$; the resulting link in $\R^3\times \{2\}$ being isotopic to the trivial link, we can cap it off in $\R^3\times \{t\}$ for some $t>4$. 
Consider likewise the trace of a link-homotopy on component $1$, introducing a singularity in $\R^3\times \{-1\}$: the resulting link is isotopic to a split link, which we may assume lies in $\R^3\times \{-2\}$; see Figure \ref{fig:link3}.  Next, take the trace of a link-homotopy on component $2$ which creates $2$ singularities, both lying in $\R^3\times \{-3\}$, and resulting in a trivial link that we can cap off in $\R^3\times \{t\}$ for some $t<-4$. 

We shall consider invariants extracted from the third component of $Y$, and  thus pick basepoints $x_i$ on $Y\cap \left(\R^3\times \{1\} \right)$, as shown in the figure ($i=1,2,3$).  
We also denote by $p_i$ ($i=1,2,3,4$) the four singular points as illustrated there; note that $\varepsilon(p_1)=\varepsilon(p_3)=-\varepsilon(p_2)=-\varepsilon(p_4)=1$. 

\begin{lemma}\label{lem:Y1}
We have $\tilde{\kappa}_Y(I;3)=0$ for any non-repeated sequence $I$. 
However, 
$$ \textrm{$\mathcal{K}_Y(1;3)=\{(1,1);(-1,1);(1;1);(-1,1)\}$\, and\, $\mathcal{K}_Y(2;3)=\{(1,-1);(1,1)\}$.} $$ 
In particular, $Y$ is not link-homotopically trivial. 
\end{lemma}

\begin{proof}
Let us first compute all $\tilde{\kappa}$ invariants extracted from the third component of $Y$. 
Let $\tilde{Y}:=Y\setminus K_3$, and denote by $G(\tilde{Y})$ the fundamental group of its exterior. By Lemma \ref{lem:isor}, we have $\nR G(\tilde{Y})\simeq \nR F(x_1,x_2)$, where $F(x_1,x_2)$ is the free group generated by $x_1$ and $x_2$. 
Pick the four based loops $\gamma_i$ ($i=1,2,3,4$) as shown in Figure \ref{fig:link3_loops}: each $\gamma_i$ sits on $Y\cap \left(\R^3\times \{1\} \right)$, and passes through the singular point $p_i$ while changing branches. 
The loop orientations shown in the figure, is the orientation induced by positivity of the associated elements of $\nR G(\tilde{Y})$, see Remark \ref{rem:orient}. 
\begin{figure}
 \includegraphics[scale=0.65]{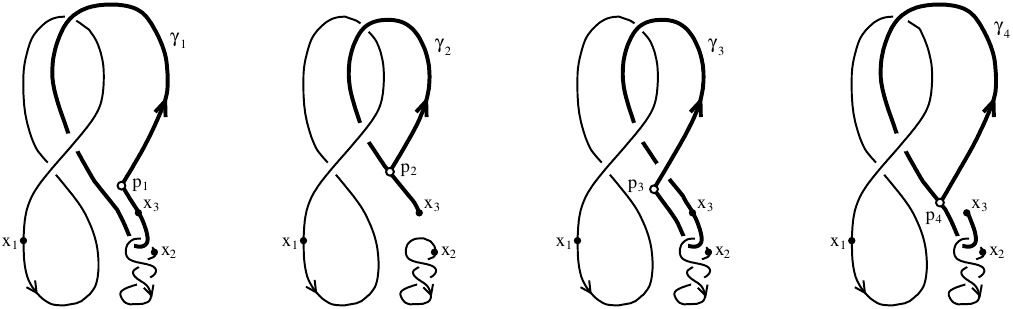}
 \caption{The based loops $\gamma_i$ ($i=1,2,3,4$)}\label{fig:link3_loops}
\end{figure}
Using the fact that the $2$-component link $\tilde{Y}\cap  \left(\R^3\times \{1\} \right)$ is isotopic to the trivial link, we  can proceed as in the proof of Proposition \ref{prop:real}, and use  Figure \ref{fig:link3_loops} to express the unique positive element $g_{i}$ of $\nR G(\tilde{Y})$ represented by the loop $\gamma_i$, for each value of $i$, as follows: 
$$ g_1=x_1x_2^{-1}\qquad ; \qquad g_2=x_1\qquad ; \qquad g_3=x_2x_1 \qquad ; \qquad g_4=x_2x_1x_2^{-1}. $$
Hence we have 
$$S_3(Y) = (x_1x_2^{-1}-1)  - (x_1-1) + (x_2x_1 - 1) - (x_2 x_1x_2^{-1} -1). $$
Taking the reduced Magnus expansion $E_3$ of each of the above four positive elements, gives 
\begin{eqnarray*}
 E_3(g_1)=1+X_1-X_2-X_1X_2 & ; & E_3(g_2)=1+X_1\\\
 E_3(g_3)=1+X_1+X_2+X_2X_1 & ; & E_3(g_4)=1+X_1+X_2X_1-X_1X_2.
\end{eqnarray*}
 It follows that $E_3(S_3(Y)) = 0$, so that $\tilde{\kappa}_Y(1;3)=\tilde{\kappa}_Y(2;3)=\tilde{\kappa}_Y(12;3)=\tilde{\kappa}_Y(21;3)=0$. 
On the other hand, the above computations of $E_3(g_i)$ ($i=1,2,3,4$) directly provide the values of $\mathcal{K}_Y(1;3)$ and $\mathcal{K}_Y(2;3)$. 
\end{proof}

\section{Closing remarks}

\subsection{Computing higher order Kirk invariants from cross sections}\label{sec:compute}

The techniques used in Section \ref{sec:ex} to compute our various examples from cross sections, 
can also be applied in a  more general setting, as follows. 

Let $f:S_1^2\cup\cdots \cup S_n^2\rightarrow \mathbb{R}^4$ be a link map and set $L=f(S_1^2\cup\cdots \cup S_n^2)$.
We may assume that all  singularities of $f$  lie in $\mathbb{R}^3\times\{0\}$, and that for each $i$ $(i=1,2,\ldots,n)$, $f^{-1}(L\cap (\mathbb{R}^3\times\{0\}))\cap S_i^2$ is an
equator $S_i^1$ of $S_i^2$. 
More precisely, we can first  deform $f$ so that all singularities and all saddle points of $f$ lie in $\mathbb{R}^3\times\{0\}$, 
all maximal points lie in $\mathbb{R}^3\times \{2\}$  
and all minimal points lie in $\mathbb{R}^3\times \{-2\}$. Then we can move 
some saddle points into $\mathbb{R}^3\times \{-1\}$ and the others into $\mathbb{R}^3\times \{1\}$, so that 
the resulting link map is of the desired form.
Then $l:=f(S_1^1\cup\cdots\cup S_n^1)$ is a self-singular link in $\mathbb{R}^3\times\{0\}$.\footnote{Here, a self-singular link in $3$-space is an immersion of circles whose singular set consists of finitely many transverse double points, each involving two strands of a same component.} 
We may further assume that $L$ has an \quote{almost product structure} in a neighborhood of the slice $\mathbb{R}^3\times \{0\}$, that is, for a sufficient small $\varepsilon>0$, we have 
\[L\cap (\mathbb{R}^3\times\{t\}))=\left\{
\begin{array}{ll}
l_+ & \textrm{, for $0<t\leq\varepsilon$}\\
l & \textrm{, for $t=0$}\\
l_- & \textrm{, for $-\varepsilon\leq t<0$,}
\end{array}
\right.\]
where $l_{\pm}:=L\cap(\mathbb{R}^3\times\{\pm\varepsilon\})$. 
We note that $l_+$ and $l_-$ are slice links, which bound slice disks
$D_+=L\cap \left(\mathbb{R}^3\times[\varepsilon,\infty)\right)$ and 
$D_-=L\cap \left(\mathbb{R}^3\times(-\infty,-\varepsilon]\right)$, respectively.
Summarizing, any link-map can be decomposed, up to deformation, into 
self-singular annuli $(l_-\times[-\varepsilon,0))\cup l\cup( l_+\times (0,\varepsilon])$ as above, capped off by 
slice disks $D_{\pm}$ for the two slice links $l_{\pm}$.

Conversely, one can construct a link-map $L$ from the following objects: 
\begin{defi}
 A \emph{slice self-singular link} is a self-singular link $l$ endowed with a sign on each singular point, such that the link $l_+$ (resp. $l_-$) obtained from $l$ by replacing each singularity 
by a classical crossing with the given sign (resp. with the opposite sign), is a slice link. 
\end{defi}

Indeed, we have the following general fact, which is a consequence of \cite[Thm.~6.10]{AMY} and \cite[Thm.~3.14]{MY7} (see \cite[Cor.~3.5]{AMY_w} for a proof). 
\begin{claim}\label{claim:slice}
Up to link-homotopy, a slice link bounds a unique union of slice disks.
\end{claim}
It follows that we have the following. 
\begin{prop}
The map $l\mapsto L$ induces a well-defined surjective map, from the set of slice self-singular links, 
to the set of link-homotopy classes of link-maps.
\end{prop}

Moreover, in the above situation, $S_i(L)$ can be calculated from the slice self-singular link $l$ as follows. 
Set $l_i:=l\setminus(\text{$i$th component of $l$})$ 
and $G(l)_i:=\pi_1((\mathbb{R}^3\times\{0\})\setminus l_i)$.
Suppose that all basepoints of $L$ lie on $l$, and that all meridians of $L$ are in  
$\mathbb{R}^3\times\{0\}$.  
Then we have the following: 
\begin{lemma}\label{lem:tek}
The inclusion $(\mathbb{R}^3\times\{0\})\setminus l\hookrightarrow R^4\setminus L$, induces
an isomorphism from $RG(l)_i$ to $RG(L)_i$. 
\end{lemma}
\noindent This result can be seen as a \lq singular\rq\, version of the \lq combinatorial Stallings Theorem\rq\, of \cite[Cor. 5.21]{AMY}. It can be shown by the same method as in the proof of \cite[Thm. 5.20]{AMY}, involving the theory of cut-diagrams. 

For each singular point $p$ of the $i$th component of $l$, 
pick a loop $\gamma_p$ passing through $p$ as in Section \ref{sec:define}, which may be assumed to lie in  $(\mathbb{R}^3\times\{0\})\setminus l_i$. 
Using Lemma \ref{lem:tek},  the unique positive  element $g_{p}$ of  $RG(L)_i$ representing $\gamma_p$, can be regarded as an element of $RG(l)_i$.
Hence in practice, one can directly compute such an element $g_{p}$, hence the sum $S_i(L)$, from a link diagram of 
$l_i\cup \gamma_p$, by using Milnor's algorithm \cite{Milnor2}  
(although  Milnor's algorithm is given for classical links in \cite{Milnor2}, it extends to self-singular links in a straightforward way, for example using the techniques of \cite{MWY}.)

\subsection{The case of surface-link maps}\label{sec:surfaces}

As seen in Section \ref{sec:cutd}, the combinatorial langage of cut-diagrams, which is the central tool in the main proofs of this paper,  
applies not only to the spherical case involved in the study of link maps, but more generally to surfaces of any topological type. 
As a matter of fact, the constructions of this paper can be generalized to link-homotopy invariants of \emph{surface-link maps}. 
Here, surface-link maps are continuous maps from a disjoint union of surfaces (of any genus) to the $4$-dimensional sphere, with pairwise disjoint images (in the case of surfaces with boundary, we require them to be properly mapped in the $4$-ball $B^4$). 
We shall not give here the details of these definitions, which follow very closely those of Section \ref{sec:define}  but only indicate the main new ingredients and specificities. 

Let $\Sigma=\Sigma_1^2\cup\cdots \cup \Sigma_n^2$ be a compact oriented surface, possibly with boundary. 
Suppose that $f:\Sigma \rightarrow B^4\subset \mathbb{R}^4$ is a surface-link map. 

Let $L=f(\Sigma_1^2\cup\cdots \cup \Sigma_n^2)$, and let $L_i=L\setminus f(\Sigma_i^2)$ for some $i$.
Then, as in the spherical case, the fundamental group $G(L_i)$ of the complement of $L_i$, is normally generated by a choice of meridian $x_j$ for each component ($j\neq i$). A presentation of the reduced group $\nR G(L_i)$ is given in \cite[Thm.~5.18]{AMY}: roughly speaking, given any set of loops $\{l_{jk}\}_k$ representing a basis for $H_1(\Sigma_j^2)$, $\nR G(L_i)$ is obtained from the reduced free group $\nR F^i_{n-1}$ on the generators $x_j$ ($j\neq i$), by adding the commuting relations $[x_j,\lambda_{jk}]$, where $\lambda_{jk}\in \nR F^i_{n-1}$ represents $l_{j,k}$. 

As a matter of fact, these extra commuting relations significantly complicate the construction given in Section \ref{sec:def} for link maps. 
On one hand, we can no longer make sense of a positive element in  $\nR G(L_i)$ as in Definition \ref{def:positive}. 
On the other hand, these relations introduce further indeterminacies in defining the element $g_p\in \nR G(L_i)$ associated with each singularity $p$ (see Notation \ref{nota:gp}), since the associated loop based at $p$ may be embedded in various ways in the surface. 
These difficulties can however be overcome, for instance as follows. 
Recall that higher-dimensional analogues of Milnor invariants were defined in \cite{AMY} for surface-links and images of surface-link maps. 
Suppose that the surface-link map $f$ has vanishing \emph{Milnor loop-invariants} $\nu_L(I)$, for all non-repeating sequences $I$ of  at most $k$ indices (see \cite[Def.~4.13]{AMY}).
This vanishing assumption ensures that the reduced Magnus expansion $E_i$ of a commutator $[x_j,\lambda_{jk}]$ as above, contains no nontrivial term of degree $<k$.
Then, the exact same construction as in Section \ref{sec:kappa} produces well-defined integers $\kappa_L(J;i)$ for any non-repeated sequence $J$ of length $<k$, which are link-homotopy invariants of $f$.

\bibliographystyle{abbrv}
\bibliography{References}
      
\end{document}